\documentclass[letterpaper,11pt]{amsart}
\usepackage[a4paper]{geometry}
\usepackage{amsmath,amssymb,amsxtra,amsthm,hyperref}
\usepackage{tikz}
\usepackage{color}


\newcommand{\tcb}[1]{\textcolor{black}{#1}}

\newcommand{\bh}[1] {\mathcal{B}(\mathcal{#1})}
\newcommand{\pref}[1] {(\ref{#1})}
\newcommand{\diag} {\operatorname{diag}}
\newcommand{\lspan} {\operatorname{span}}

\newcommand{\cH}{\mathcal{H}}
\newcommand{\cK}{\mathcal{K}}
\newcommand{\cA}{\mathcal{A}}
\newcommand{\cB}{\mathcal{B}}
\newcommand{\cD}{\mathcal{D}}

\newcommand{\cN}{\mathcal{N}}
\newcommand{\cO}{\mathcal{O}}
\newcommand{\cQ}{\mathcal{Q}}

\newcommand{\cX}{\mathcal{X}}

\newcommand{\bF}{\mathbb{F}}

\def\mult{\operatorname{Mult}}
\newtheorem{theorem}{Theorem}[section]
\newtheorem{corollary}[theorem]{Corollary}
\newtheorem{proposition}[theorem]{Proposition}    %%%%%%%%%%%%%%%%%%%%%%%%%5
\newtheorem{lemma}[theorem]{Lemma}

\theoremstyle{definition}
\newtheorem{definition}[theorem]{Definition}

\newtheorem{example}[theorem]{Example}
\newtheorem{remark}[theorem]{Remark}

\numberwithin{equation}{section}

\begin{document}

\title[Dilation and Semigroup Dynamical System]{Dilation theory for right LCM semigroup \\dynamical systems}

\author{Marcelo Laca}
\address{Department of Mathematics and Statistics, University of Victoria, Victoria, B.C. V8W 3R4}
\email{laca@uvic.ca}

\author{Boyu Li}
\address{Department of Mathematics and Statistics, University of Victoria, Victoria, B.C. V8W 3R4}
\email{boyuli@uvic.ca}

\date{\today}

\thanks{\tcb{This research was supported by the Natural Sciences and Engineering Research Council of Canada Discovery Grant RGPIN-2017-04052.}\ The second author was supported by a \tcb{postdoctoral} fellowship of \tcb{the} Pacific Institute for the Mathematical Sciences.}


\begin{abstract} This paper examines actions of right LCM semigroups by endomorphisms of C*-algebras that encode an additional structure of the right LCM semigroup. We define contractive covariant representations for these semigroup dynamical systems and prove a generalized Stinespring's dilation theorem showing that these representations  can be dilated if and only if the map on the $\mathrm{C}^*$-algebra is unital and completely positive. This generalizes earlier results about dilations of right LCM semigroups of contractions. In addition, we also give sufficient conditions under which a contractive covariant representation of a right LCM system can be dilated to an isometric representation of the boundary quotient.
\end{abstract}

\maketitle

\section{Introduction}
A contractive representation of a semigroup $P$ on a Hilbert space $\cH$ is a semigroup homomorphism $T:P\to\bh{H}$ such that $\|T_p\| \leq 1$ for every $p\in P$. If $p$ is not invertible in $P$ the operator $T_p$ does not have to be isometric. For example, any contractive operator $A\in\bh{H}$ gives rise to a contractive representation $T:\mathbb{N}\to\bh{H}$ by $T(n)=A^n$. Nevertheless, a celebrated theorem of Sz.\ Nagy's  shows that a contractive representation of $\mathbb{N}$ is always the compression of an isometric representation. In other words, there exists an isometry $V$ on a larger Hilbert space $\cK\supset\cH$, such that 
\[ A^n = P_{\cH} V^n \big|_\cH \qquad (n\in\mathbb{N}).\]
The operator $V$ is called an isometric dilation of $T$.  Nagy's dilation theorem has been generalized 
in several directions. Examples include Ando's dilation for commuting contractions \cite{Ando1963}, Brehmer's regular dilation \cite{Brehmer1961}, and Frazho-Bunce-Popescu's dilation of row contractions \cite{Frazho1982, Bunce1984, Popescu1989}.

In his seminal work on $\mathrm{C}^*$-algebras of quasi-lattice ordered semigroups \cite{Nica1992}, Nica abstracts a notion of covariance (Nica-covariance) from the left-regular representation and uses it to define a sensible notion of universal semigroup $\mathrm{C}^*$-algebra. The study of Nica-covariant representations  transformed the study of semigroup $\mathrm{C}^*$-algebras into a very active area, and guided its development until the inception of a vast generalization in work of X. Li \cite{XLi2012}. 
%,XLi2014,CELY}. 
Early developments include the semigroup crossed product techniques from \cite{LacaRaeburn1996}, applications to Artin monoids \cite{CrispLaca2002,CrispLaca2007}; more recent ones include the analysis of right LCM semigroups \cite{BLS2017, Starling2015}.
Meanwhile, dilation theory also developed as an important tool in connecting various types of semigroup representations, as, for example,  dilation theorems for lattice ordered semigroups \cite{Fuller2013,BLi2014} and right LCM semigroups \cite{Li2017, BLi2019}.

The study of semigroup representations and semigroup $\mathrm{C}^*$-algebras is closely related to the study of semigroup dynamical systems. For example, the Cuntz algebra  $\cO_n$ can be realized as a crossed product of a certain semigroup dynamical system over a UHF algebra of type $n^\infty$. Moreover, as pointed out in \cite{Nica1992} and  formally established in \cite{LacaRaeburn1996}, $\mathrm{C}^*$-algebras of quasi-lattice ordered semigroups also possess a crossed product structure, of a canonical action of  the semigroup  on the commutative diagonal algebra. This construction also motivated the study of boundary quotients, an analogue of the Cuntz algebras in the realm of semigroup $\mathrm{C}^*$-algebras. On the other hand, in the realm of non-self-adjoint operator algebras, semigroup dynamical systems play a central role in the construction of various semicrossed product algebras and the computation of their $\mathrm{C}^*$-envelope \cite{EvgElias2012, DK2012, DFK2014, Li2020env}. Dilation theory is often a useful tool in these studies. 

In this paper, we consider a class of dynamical systems over cancellative right LCM semigroups,  which we call {\em right LCM dynamical systems}. In general, the semigroup action $\alpha$ for a semigroup dynamical system only encodes the semigroup multiplication via $\alpha_p\alpha_q=\alpha_{pq}$, and any additional structure of the semigroup is often lost.  For example, even if the semigroup has a lattice structure, there is no reason to expect that this structure will be reflected by a semigroup action, although in some cases it is indeed automatically preserved, e.g. the action by endomorphisms that lead to the Bost-Connes algebra. It was the analysis of this specific example that led to the consideration of  semigroup actions that ``respect the lattice structure" \cite[Definition 3]{Laca1998}, and our first step is to extend this definition to actions by right LCM semigroups, see Definition \ref{df.respectLCM} below. 

In the study of semigroup dynamical systems $(\cA,P,\alpha)$, an isometric covariant representation is a pair $(\pi,V)$ where $\pi$ is a $*$-representation of the $\mathrm{C}^*$-algebra $\cA$,  $V$ is an isometric representation of the semigroup $P$, and the $P$-action $\alpha$ is encoded by the covariance relation:
\[V(p) \pi(a) V(p)^* = \pi(\alpha_p(a)).\]
We can similarly define a contractive covariant representation to be a pair $(\phi, T)$ where $\phi$ is a unital $*$-preserving linear map, which need not   be multiplicative, and $T$ is a contractive representation of the semigroup $P$. Again, the $P$-action is encoded by the covariance relation: 
\[T(p) \phi(a) T(p)^* = \phi(\alpha_p(a)).\]
Our main result, Theorem \ref{thm.dym.dilation}, states that a contractive covariant representation $(\phi,T)$ can be dilated to an isometric covariant representation $(\pi,V)$ if and only if $\phi$ is a unital completely positive map. In the special case when the semigroup is trivial, i.e. $P=\{e\}$, our theorem recovers the celebrated Stinespring's dilation theorem for unital completely positive maps. 

The  paper is organized as follows. In Section \ref{sec.prelim} we give a brief overview of semigroups and their representations. In Section \ref{sect.LCMDynSys} we define  right LCM semigroup dynamical systems and their isometric and \tcb{contractive} covariant representations, Definitions \ref{df.CovariantPair} and \ref{df.ContractiveCovariantPair}. In order to study dilations on semigroup dynamical systems  in \tcb{S}ection \ref{sect.Naimark}, we first define the notion of kernel systems, Definition \ref{df.kernel} and then develop a general framework to deal with dilations of our semigroup dynamical systems, Theorem \ref{thm.kernel.dilation}. This can be viewed as a generalization of the Naimark dilation theorem for positive definite Toeplitz kernels, a major tool in the study of dilation theory of semigroup representations \cite{Popescu1999b, BLi2019}. In section \ref{sect.dilation} we prove our main result, Theorem \ref{thm.dym.dilation}, by first building a kernel system $K$ from a contractive covariant representation $(\phi,T)$, Definition \ref{df.K}, and then proving that this kernel system is positive if and only if $\phi$ is a unital completely positive map, Proposition \ref{prop.positive.ucp}. The proof is then concluded by showing that when $\phi$ is unital completely positive, the minimal Naimark dilation for $K$  \tcb{is} an isometric covariant representation. 
One major motivation for this paper comes from the dynamical systems arising from semigroup $\mathrm{C}^*$-algebras and their boundary quotients, and we explore such connections in section \ref{sect.exLCM}. As a highlight, we characterize all representations that can be dilated to a representation of the boundary quotient, Corollary \ref{cor.QDilation}. Finally, in section \ref{sect.example}, we explore some additional constructions of right LCM semigroup dynamical systems and their dilation theory, yielding several new results on dilating unital completely positive maps satisfying certain relations.

\section{Preliminaries}\label{sec.prelim}

\subsection{Semigroups and representations} 

A semigroup $P$ is a set with an associative multiplication. An element $e\in P$ is called the identity if $ex=xe=x$ for all $x\in P$. A semigroup with an identity is called a monoid. An element $p\in P$ need not have an inverse; when it does, it is called \emph{a unit}. The set of all units is denoted by $P^*$, which is a group on its own. 

A semigroup $P$ is called \emph{left-cancellative} if for any $p,x,y\in P$ with $px=py$, we must have $x=y$. We can similarly define the notion of right-cancellative, and we say a semigroup is \emph{cancellative} if it is both left and right cancellative. It is often convenient to assume that the semigroup $P$ embeds inside a group $G$, in which case $P$ must be cancellative. However, the converse is false: cancellative semigroups need not embed in groups, and it is often a challenging task to verify whether a semigroup can be embedded in a group. For example,  it was only relatively recently that all Artin monoids were shown to embed in their corresponding Artin groups \cite{Paris2002}. Throughout this paper, unless stated otherwise, we assume the semigroup $P$ is cancellative and contains an identity $e$.

Nica first defined and studied quasi-lattice ordered groups in \cite{Nica1992}. Consider a semigroup $P$ inside a group $G$ with $P\cap P^{-1}=\{e\}$. The semigroup $P$ defines a partial order on $G$ by $x\leq y$ when $x^{-1}y\in P$. This order is called a \emph{quasi-lattice order} if, for any finite $F\subset G$ with a common upper bound, there exists a least common upper bound. The pair $(G,P)$ is called a \emph{quasi-lattice ordered group}, and we often refer to $P$ as a \emph{quasi-lattice ordered semigroup}. 

Nica's work has been extended to more general classes of semigroups, for example, right LCM semigroups \cite{BLS2017}, 
%and weak quasi-lattice ordered semigroups \cite{ABCD2019}.
which are a large class of semigroups that closely resemble  quasi-lattice ordered semigroups. 
A left-cancellative semigroup $P$ is called a \emph{right LCM semigroup} if for any $p,q\in P$, either $pP\cap qP=\emptyset$ or else $pP\cap qP=rP$ for some $r\in P$. Here, the element $r$ can be seen as a least-common upper bound of $p,q$. Since we do not preclude non-trivial units in $P$, it is possible that the choice of $r$ is not unique. Nevertheless, for $r,s\in P$, $rP=sP$ precisely when $r=su$ for some unit $u\in P^*$. It is easy to see that quasi-lattice ordered semigroups, and more generally the weak quasi-lattice ordered semigroups from \cite{ABCD2019}, are right LCM, but the converse does not hold. 

%\begin{example} 
Quasi-lattice ordered and right LCM semigroups cover a wide range of familiar semigroups. Subsemigroups of $\mathbb{R}^+$, the free semigroup $\mathbb{F}_n^+$, right-angled Artin monoids $A_\Gamma^+$ are all quasi-lattice ordered, and hence right LCM. Finite-type Artin monoids are in fact lattice ordered semigroups \cite{BrieskornSaito1972}. However, it is unknown whether all Artin monoids are quasi-lattice ordered semigroups inside their corresponding Artin groups. Nevertheless, since for right LCM semigroups the least common upper bound is only required for semigroup elements, we do know that Artin monoids are right LCM \cite[Proposition 4.1]{BrieskornSaito1972}, in fact, they are weak quasi-lattice ordered. 
%\end{example}  

A \emph{contractive representation} $T$ of a semigroup $P$ is a unital semigroup homomorphism $T:P\to\bh{H}$ such that each $T(p)$ is a contraction. The representation is called isometric (resp.\ unitary) if each $T(p)\in\bh{H}$ is an isometry (resp.\ a unitary operator).  An elementary argument shows that when $u$ is invertible in $P$, then $T(u)$ must be unitary.  So the restriction of a contractive representation to the group of units $P^*$ is a
unitary representation. 

Just like with groups, an important representation of a semigroup is its left regular representation 
$\lambda:P\to\cB(\ell^2(P))$  defined by $\lambda(p)\delta_q=\delta_{pq}$ on the standard orthonormal basis $\{\delta_p\}_{p\in P}$. It follows from the left-cancellation that $\lambda(p)$ maps the orthonormal basis to an orthonormal set, and thus $\lambda$ is an isometric representation. The reduced semigroup $\mathrm{C}^*$-algebra of $P$, often denoted as $\mathrm{C}^*_\lambda(P)$ is the {$\mathrm{C}^*$-algebra} generated by the image of $\lambda$. 

One might be tempted to define the semigroup $\mathrm{C}^*$-algebra as the universal $\mathrm{C}^*$-algebra with respect to isometric representations of the semigroup. Murphy proved that under such \tcb{a} definition, the universal object for isometric representations of $\mathbb{N}^2$ is not nuclear, which is undesirable as the universal object for an abelian semigroup as simple as $\mathbb{N}^2$. The issue is resolved by requiring that the range projections commute. This was abstracted and generalized by Nica, leading to what is now known as the Nica-covariance condition. In terms of a right LCM semigroup $P$, an isometric representation $V:P\to\bh{H}$ is called \emph{Nica-covariant} if for any $p,q\in P$, 
\[V(p)V(p)^*V(q)V(q)^* = 
\begin{cases}
V(r)V(r)^*, &\mbox{ if } pP\cap qP=rP\tcb{;} \\
0, &\mbox{ if } pP\cap qP=\emptyset\tcb{.}
\end{cases}
\]
%Made it consistant that we are using ";" at the end of first case and "." at the end of the second case.

One can easily verify that, for the left-regular representation $\lambda$, the range projection $\lambda(p)\lambda(p)^*$ is the orthogonal projection onto the subspace $\ell^2(pP)$, and thus satisfies the Nica-covariance condition. For a right LCM semigroup $P$, the universal $\mathrm{C}^*$-algebra of the isometric Nica-covariant representations is called the semigroup $\mathrm{C}^*$-algebra, which is often denoted as $\mathrm{C}^*(P)$. 

\begin{example} Consider the semigroup $P=\mathbb{N}^2$, which is the free abelian semigroup generated by two generators $e_1,e_2$. An isometric representation $V:\mathbb{N}^2\to\bh{H}$ is uniquely determined by its value on two commuting generators $V_i=V(e_i)$, $i=1,2$. Here, $V_1, V_2$ must commute because $e_1,e_2$ commute.

Since $e_1P\cap e_2P=(e_1+e_2)P$, the Nica-covariance condition requires that \[V_1V_1^* V_2V_2^* = V_1V_2 (V_1V_2)^*.\]
One can easily check that this is equivalent to requiring that $V_1$ and $V_2$ are $\ast$-commuting, i.e., $V_1$ commutes with both $V_2$ and $V_2^*$. 
\end{example} 

\begin{example} Consider the free semigroup $P=\mathbb{F}_n^+$, with $n$ free generators $e_1,\cdots,e_n$. An isometric representation $V:\mathbb{F}^+_n\to\bh{H}$ is uniquely determined by $n$ non-commuting isometries $V_i=V(e_i)$. 

Since $e_i P\cap e_j P=\emptyset$ for all $i\neq j$, the Nica-covariance condition requires that $V_i V_i^*$ are pairwise orthogonal projections. This is often characterized by a single condition $\sum_{i=1}^n V_iV_i^*\leq I$.  
\end{example}

\subsection{Semigroup dynamical systems}

Suppose $\cA$ is a unital $\mathrm{C}^*$-algebra and $P$ is a semigroup. A $P$-action on $\cA$ is a map $\alpha$ that sends each $p\in P$ to a $\ast$-endomorphism $\alpha_p:\cA\to\cA$ that satisfies $\alpha_p\circ \alpha_q=\alpha_{pq}$ and $\alpha_e=id$. Throughout this paper, we always assume that each $\alpha_p$ is injective because this plays a crucial role in our constructions. The action $\alpha$ is called automorphic if each $\alpha_p$ is a $\ast$-automorphism of $\cA$. Whenever $p$ is invertible, $\alpha_p$ must be a $\ast$-automorphism. 

A \emph{semigroup dynamical system} is a triple $(\cA,P,\alpha)$ where $\cA$ is a unital $\mathrm{C}^*$-algebra, $P$ is a left-cancellative semigroup\tcb{, and $\alpha$ is a $P$-action on $A$}. When each $\alpha_p$ is injective, we say that  $\alpha$ is an injective $P$-action on $\cA$. For each $p\in P$, we denote $\cA_p=\alpha_p(\cA)$, which is a $\mathrm{C}^*$-subalgebra of $\cA$. This subalgebra has a unit $\alpha_p(1)$. Since $\alpha$ is a $\ast$-endomorphism, $\alpha_p(1)$ is always an orthogonal projection. Moreover, certain ordering on the semigroup is preserved by $\alpha_p(1)$.

\begin{lemma}\label{lm.LCM.N.1}  If  $(\cA,P,\alpha)$ is a semigroup dynamical system, then $\alpha_x(1)\alpha_y(1)=\alpha_y(1)$ for $x,y,p\in P$ with $y=xp$.
\end{lemma}

\begin{proof} $\alpha_x(1)\alpha_y(1)=\alpha_x(1)\alpha_x(\alpha_p(1))=\alpha_x(\alpha_p(1))=\alpha_y(1).\qedhere$
\end{proof} 

\begin{remark} If $P$ is a right LCM semigroup with $P^* =\{e\}$, then one can define a partial order by $x\leq y$ whenever there exists a $p\in P$ with $y=xp$. In this case Lemma \ref{lm.LCM.N.1} proves that  $\alpha_x(1) \geq \alpha_y(1)$ whenever $x\leq y$.
\end{remark}

If $p\in P^*$, then $\cA_p=\cA$ since $\alpha_p$ is a $\ast$-automorphism. If $pP=qP$, then $p=qu$ for some $u\in P^*$, in which case $\cA_p=\cA_q$. 

Semigroup dynamical systems occur naturally in many well-known $\mathrm{C}^*$-algebraic constructions. For example, the Cuntz algebra $\cO_n$ can be constructed from a semigroup dynamical system via a single $\ast$-endomorphism $\alpha$ acting on a UHF algebra of type $n^\infty$. The single $\ast$-endomorphism induces a $\mathbb{N}$-action on the UHF algebra. 
An important class of semigroup dynamical systems that are central to this paper arises from considering  semigroup $\mathrm{C}^*$-algebras.

\begin{example} Let $P$ be a right LCM semigroup and $V:P\to\bh{H}$ be a universal  Nica-covariant isometric  representation of $P$. The diagonal algebra is defined by \[\cD_P=\overline{\lspan}\{V(p)V(p)^*: p\in P\},\]
and is a commutative $\mathrm{C}^*$-algebra because the Nica-covariance condition implies that $V(p)V(p)^*V(q)V(q)^*$ is either $0$ or $V(r)V(r)^*$ for some $r\in P$,
and the set of such $r$ only depends on the intersection $pP \cap qP$. The semigroup $P$ acts on $\cD_P$ naturally via $\alpha_p:\cD_P\to\cD_P$, where $\alpha_p(x)=V(p)xV(p)^*$. We shall study this dynamical system in more detail in section \ref{sect.exLCM}. 
\end{example}  

\section{Right LCM Dynamical System}\label{sect.LCMDynSys}

\subsection{The LCM condition}

%In a semigroup dynamical system $(\cA,P,\alpha)$, only the multiplicative structure of the semigroup is encoded via $\alpha_{pq}=\alpha_p\alpha_q$. Any additional structure that $P$ may have may be lost, and in order to preserve it one needs to add extra conditions on the action. 
%An early example of a condition that encodes the lattice structure for abelian lattice ordered semigroups in the dynamical systems was considered in \cite{Laca1998}.
%\begin{definition}\cite[Definition 3]{Laca1998}\label{df.respectLattice} Let $P$ be an abelian lattice ordered semigroup. We say that a $P$-action $\alpha$ on a $\mathrm{C}^*$-algebra $\cA$ \emph{respects the lattice structure} if each $\alpha_p(\cA)$ is an ideal of $\cA$, and \[\alpha_p(1)\alpha_q(1)=\alpha_{p\vee q}(1).\] Here, $p\vee q$ is the least common upper bound of $p,q$ with respect to the lattice structure of $P$.
%\end{definition} 
 Recall that if $p,q$ are two elements in a right LCM semigroup $P$ and if $pP\cap qP=rP$,  then $r$ behaves like a least common upper bound of $p$ and $q$. In this paper we focus on dynamical systems that respect this feature of right LCM semigroups in a sense made precise in the following definition that generalizes \cite[Definition 3]{Laca1998}.

\begin{definition}\label{df.respectLCM} Let $P$ be a right LCM semigroup. An injective $P$-action $\alpha$ on a $\mathrm{C}^*$-algebra $\cA$ \emph{respects the right LCM} if each $\alpha_p(\cA)$ is an ideal of $\cA$, and for any $p,q\in P$,
\[ \alpha_p(1) \alpha_q(1) = \begin{cases}
\alpha_r(1), &\mbox{ if } pP\cap qP=rP; \\
0, &\mbox{ if } pP\cap qP=\emptyset.
\end{cases} \]
A  \emph{right LCM dynamical system} is a dynamical system $(\cA,P,\alpha)$ in which $\alpha$ respects the right LCM. 
\end{definition} 

Even though the choice of $r$ such that $rP=pP\cap qP$ may not be unique, the condition introduces no ambiguity. The reason is that when $rP=sP$, then $r=su$ for some unit $u\in P^*$,  hence $\alpha_u$ is a $\ast$-automorphism so that $\alpha_u(1)=1$ and  $\alpha_r(1)=\alpha_s(1)$. For each $p\in P$ we denote $E_p :=\alpha_p(1)$ which is clearly a projection and is the identity of the ideal $\cA_p :=\alpha_p(\cA) =  E_p \cA E_p$.
The right LCM condition manifests itself in a right LCM dynamical system in many ways. One of them is through the intersection of ideals $\{\cA_p\}$.

\begin{proposition} Let $(\cA,P,\alpha)$ be a right LCM dynamical system. Then for every $p,q\in P$, 
\[\cA_p \cap \cA_q = \begin{cases} 
\cA_r, &\mbox{ if } pP\cap qP=rP; \\
\{0\}, &\mbox{ if } pP\cap qP=\emptyset.
\end{cases}\]
\end{proposition} 

\begin{proof} If $a\in \cA_p\cap \cA_q$, then $aE_pE_q= aE_q=a$. 
If $pP\cap qP=\emptyset$, then $E_pE_q=0$ and thus $a=0$. 
If $pP\cap qP=rP$, then $E_pE_q=E_r$ and thus $a\in \cA_r$, proving $\cA_p\cap \cA_q \subset \cA_r$.
To prove the reverse inclusion, suppose $r\in pP\cap qP$ and write $r=pp_0=qq_0$ for some $p_0,q_0\in P$,
 so $\cA_r = \alpha_{pp_0}(\cA) \subset \alpha_p(\cA) = \cA_p$ and similarly for $q$. This proves $\cA_r\subseteq \cA_p \cap \cA_q$. 
\end{proof}

\begin{proposition}\label{prop.equivLCM} 
Let $P$ be a right LCM monoid. A dynamical system $(\cA, P, \alpha)$ is a right LCM dynamical system if and only if for every $p,q\in P$, 
\begin{equation}\label{eq:rightlcmsyst-charact}
\cA_p \cA_q := \{ab: a\in \cA_p, b\in \cA_q\}=\begin{cases} \cA_r, &\mbox{ if } pP\cap qP=rP;\\
\{0\}, &\mbox{ if } pP\cap qP=\emptyset.
\end{cases} 
\end{equation}
\end{proposition} 

\begin{proof} Suppose $(\cA, P, \alpha)$ is a right LCM dynamical system.
 For any $a,b\in\cA$, consider $\alpha_p(a)\alpha_q(b)=\alpha_p(a) E_pE_q \alpha_p(b)$. When $pP\cap qP=\emptyset$, $E_pE_q=0$ and $\alpha_p(a)\alpha_q(b)=0$, and hence $\cA_p\cA_q=\{0\}$. When $pP\cap qP=rP$, $E_pE_q=E_r \in\cA_r$, and $\alpha_p(a)E_r\alpha_p(b)\in \cA_r$ since $\cA_r$ is an ideal. Therefore, $\cA_p\cA_q\subseteq \cA_r$. To see the other inclusion, take any $\alpha_r(a)\in \cA_r$, since $E_r=E_pE_q$, we have $\alpha_r(a)=\alpha_r(a) E_p E_q$, where $\alpha_r(a) E_p\in \cA_p$ and $E_q\in \cA_q$. Hence, $\cA_p\cA_q=\cA_r$.

To prove the converse, suppose $(\cA, P, \alpha)$ satisfies  equation  \eqref{eq:rightlcmsyst-charact}.
 Let $p\in P$ and set $q=e\in P$. We have $\cA_e=\cA$. Since $pP\cap eP=pP$, we must have $\cA_p\cA \subseteq \cA_p$ and $\cA\cA_p\subseteq \cA_p$. This implies that $\cA_p$ is an ideal. For any $p,q\in P$, it suffices to prove that $E_p=\alpha_p(1)$ satisfies the right LCM condition. If $pP\cap qP=\emptyset$, then $E_pE_q\in \cA_p\cA_q=\{0\}$ so that $E_pE_q=0$. If $pP\cap qP=rP$, then $E_pE_q\in \cA_p\cA_q\subseteq \cA_r$. Therefore, $E_pE_q=\alpha_r(a)$ for some $a\in \cA$. Since $r=pp_0=qq_0$ for some $p_0,q_0\in P$, $\cA_r\subseteq \cA_p,\cA_q$. Hence, for any $\alpha_r(b)\in \cA_r$,
\[\alpha_r(ab)=E_pE_q\alpha_r(b)=\alpha_r(b)=\alpha_r(b)E_pE_q=\alpha_r(ba).\]
Since $\alpha_r$ is injective, $a$ must be the identity and $E_p E_q=E_r$. Therefore, $\alpha$ respects the right LCM. 
\end{proof} 

\begin{remark} One has to be cautious that our definition of the product $\cA_p\cA_q$ does not correspond to the product ideal of two ideals, which is usually defined to be the closure of the linear span of products and satisfies
\[\overline{\lspan}\{ab:a\in I, b\in J\}=I\cap J.\]
In general, the product set $\{ab: a\in I, b\in J\}$ is not  an ideal, so  it is crucial for Proposition \ref{prop.equivLCM} that we take advantage of the right LCM condition of our semigroup dynamical systems. 
\end{remark} 

As an immediate corollary, we obtain the following factorization property.

\begin{corollary} For a right LCM dynamical system and  $p,q\in P$,
\[\cA_p\cA_q = \cA_p\cap \cA_q.\]
\end{corollary} 

Even though each $\alpha_p$ is not a $\ast$-automorphism, our injectivity assumption ensures that it does have a left-inverse, defined on all of $\cA$. 

\begin{proposition}\label{prop.left.inv} For each $p\in P$, let $\alpha_p^{-1} : \cA_p\to \cA$ be the inverse of $\alpha_p$ and define 
$\alpha_{p^{-1}}:\cA\to\cA$ by \[\alpha_{p^{-1}}(a) =  \alpha_p^{-1}(\alpha_p(1) a)\]
Then $\alpha_{p^{-1}}$ is well defined,  $\alpha_{p^{-1}} \circ \alpha_p=id$, and $\alpha_p \alpha_{p^{-1}} = \mult_{\alpha_p(1)}$ is the multiplication operator by $\alpha_p(1)$.
\end{proposition}

\begin{proof} Since $\cA_p$ is an ideal, $\alpha_p(1)a\in\cA_p$. Since $\alpha_p$ is injective, $\alpha_{p^{-1}}(a)$ is well defined. For any $a\in A$, \[\alpha_{p^{-1}}(\alpha_p(a))=\alpha_p^{-1}(\alpha_p(1)\alpha_p(a))=\alpha_p^{-1}(\alpha_p(a))=a,\]
and,
\[\alpha_{p}(\alpha_{p^{-1}}(a))=\alpha_p(\alpha_p^{-1}(\alpha_p(1)a))=\alpha_p(1) a. \qedhere\]

\end{proof}

The map $\alpha_{p^{-1}}$ has many nice properties.

\begin{proposition}\label{prop.left.inv2} For a right LCM dynamical system $(\cA,P,\alpha)$ and for every $p, q \in P$,
and $a\in \cA$ we have
\begin{enumerate}
\item  $\alpha_{p^{-1}}(a)=\alpha_p^{-1} (a \alpha_p(1))=\alpha_p^{-1} (\alpha_p(1) a \alpha_p(1))$;
\smallskip\item  $\alpha_{p^{-1}}$ is a surjective $\ast$-endomorphism on $\cA$; and 
\smallskip\item $\alpha_{(pq)^{-1}}(a)=\alpha_{q^{-1}}(\alpha_{p^{-1}}(a)).$
\end{enumerate}
\end{proposition} 

\begin{proof} For (1), since $\alpha_p(\cA)$ is an ideal of $\cA$, the elements $\alpha_p(1)a$ and $a\alpha_p(1)$ are in $ \alpha_p(\cA)$, and therefore,
\[\alpha_p(1)a = \alpha_p(1)a\alpha_p(1) = a\alpha_p(1).\]
Since $\alpha_p$ is injective, the inverses under $\alpha_p$ of these three terms are all equal to $\alpha_{p^{-1}}(a)=\alpha_p^{-1}(\alpha_p(1)a)$. 

For (2), it is easy to verify that $\alpha_{p^{-1}}$ is a $\ast$-preserving linear map since $\alpha_p$ is. To see it is multiplicative, let $a,b\in \cA$; then
\[\alpha_p (\alpha_{p^{-1}}(ab)) = \alpha_p(1) a b \]
On the other hand \[\alpha_p (\alpha_{p^{-1}}(a)\alpha_{p^{-1}}(b)) = \alpha_p(1) a \alpha_p(1) b = \alpha_p(1) a  b,\]
which gives the result because $\alpha_p$ is injective
 %$\alpha_p(1)a, b\alpha_p(1)\in \alpha_p(\cA)$, this becomes \[\alpha_p^{-1}(\alpha_p(1) a b \alpha_p(1)) = \alpha_p^{-1}(\alpha_p(1) a) \alpha_p^{-1}( b \alpha_p(1))=\alpha_{p^{-1}}(a)\alpha_{p^{-1}}(b).\]

For (3), recall notice that $\alpha_{pq}(1)=\alpha_{pq}(1)\alpha_p(1)$, and $\alpha_p(1)(a)=\alpha_p(\alpha_{p^{-1}}(a))$, and compute
\begin{align*}
\alpha_{(pq)^{-1}}(a) &= \alpha_{pq}^{-1} (\alpha_{pq}(1)a) \\
&= \alpha_q^{-1}(\alpha_p^{-1}(\alpha_{pq}(1)\alpha_p(1) a)) \\
&= \alpha_q^{-1}(\alpha_p^{-1}(\alpha_p( \alpha_q(1) \alpha_{p^{-1}}(a)))) \\
&= \alpha_q^{-1}(\alpha_q(1) \alpha_{p^{-1}}(a) ) \\
&= \alpha_{q^{-1}}(\alpha_{p^{-1}}(a)).\qedhere
\end{align*}
\end{proof}

\begin{example} A right LCM semigroup $P$ is called \emph{semi-lattice ordered} if $pP\cap qP\neq \emptyset$  for every $p,q\in P$. If a semi-lattice ordered monoid $P$ acts by $\ast$-automorphisms of $\cA$, then 
$\alpha_p(1)=1$ and $\cA_p=\cA$ for all $p\in P$, so the system $(\cA, P, \alpha)$ is always a right LCM dynamical system. 
\end{example} 

\begin{example} Let $X$ be a compact Hausdorff space and suppose that there exists a collection of clopen subsets $\{X_p\}_{p\in P}$ and a $P$-action $\beta$ on $X$ such that each $\beta_p:X\to X_p$ is a homeomorphism from $X$ onto $X_p$. This induces a $P$-action $\alpha$ on the unital commutative $\mathrm{C}^*$-algebra $C(X)$, where $\alpha_p:C(X)\to C(X_p)\subseteq C(X)$ is given by
\[\alpha_p(f)=f\circ \beta_p^{-1}\]

The action $\alpha$ is an LCM action if $X_p$ satisfies $X_p\cap X_q\subseteq X_r$ if $pP\cap qP=rP$, and $X_p\cap X_q=\emptyset$ if $pP\cap qP=\emptyset$. Indeed, we have $\alpha_p(C(X))=C(X_p)$ and thus $\alpha_p(C(X))\alpha_q(C(X))=C(X_p\cap X_q)$. 
\end{example}

\begin{example} Unfortunately, not all dynamical systems respect the right LCM. For example, the Cuntz algebra $\cO_n$ is closely related to the semigroup dynamical system $(\cA, \mathbb{N}, \alpha)$ where $\cA$ is the UHF algebra of type $n^\infty$ and the action $\alpha$ is determined by $\alpha_1(a)=e_{11}\otimes a$. This is not a right LCM dynamical system because the UHF algebra is a simple $\mathrm{C}^*$-algebra, and so the proper subalgebra $\alpha_1(\cA)$ cannot be an ideal.  
\end{example} 

\subsection{Dynamical systems over $\mathbb{F}_k^+$ and AF-algebras}

It is clear that if $(\cA,\mathbb{F}_k^+,\alpha)$ is an injective dynamical system, then the action $\alpha$ is uniquely determined by the endomorphisms  $\alpha_1, \cdots, \alpha_k$ associated to the generators. 
We would like to use this to 
provide  a simple characterization of right LCM semigroup dynamical systems over the free semigroup $\mathbb{F}_k^+$.

\begin{proposition}\label{prop.LCM.N} An injective semigroup dynamical system $(\cA,\mathbb{F}_k^+,\alpha)$ is a right LCM dynamical system if and only if for each $1\leq i\leq k$, $\alpha_i(\cA)$ is an ideal of $\cA$, and for all $1\leq i,j\leq k$ with $i\neq j$, $\alpha_i(1)\alpha_j(1)=0$.
\end{proposition} 

\begin{proof}
First assume that $\alpha_i(\cA)$ is an ideal of $\cA$ for all $1\leq i\leq k$. We first prove that  $\alpha_w(\cA)$ is an ideal of $\cA$ for every word $w=w_1w_2\cdots w_n\in\mathbb{F}_k^+$.  Since $\alpha_i(\cA)$ is an ideal of $\cA$ and $\alpha_i$ is injective, we can define $\alpha_{i^{-1}}(x)=\alpha_i^{-1}(\alpha_i(1)x)$ which is a left-inverse for $\alpha_i$. If $a,b\in \cA$, then
\begin{align*}
\alpha_w(a) b &= \alpha_{w_1}(\alpha_{w_2w_3\cdots w_n}(a)) \alpha_{w_1}(1) b \\
&= \alpha_{w_1}(\alpha_{w_2w_3\cdots w_n}(a)) \alpha_{w_1}(\alpha_{w_1^{-1}}(b))  \\
&= \alpha_{w_1}(\alpha_{w_2w_3\cdots w_n}(a)\alpha_{w_1^{-1}}(b)).
\end{align*}

Letting $b_1:=\alpha_{w_1^{-1}}(b)$ we can use a similar argument to show that
\begin{align*}
\alpha_w(a) b &= \alpha_{w_1}(\alpha_{w_2w_3\cdots w_n}(a)b_1) \\
&= \alpha_{w_1w_2}(\alpha_{w_3\cdots w_n}(a) \alpha_{w_2^{-1}}(b_1)).
\end{align*}

Repeating this process, we see that $\alpha_w(a) b\in\alpha_w(\cA)$. By the Proposition \ref{prop.left.inv2}, the same argument can be made for multiplication by $b$ from the left. This proves that $\alpha_w(\cA)$ is an ideal. 

Now, for all $1\leq i,j\leq k$ with $i\neq j$, $\alpha_i(1)\alpha_j(1)=0$. For any two elements  $x,y\in\mathbb{F}_k^+$, we write $x=x_1x_2\cdots x_n$ and $y=y_1y_2\cdots y_m$. Without loss of generality, we may assume that $n\leq m$. There are two possibilities: if $x\mathbb{F}_k^+\cap y\mathbb{F}_k^+\neq\emptyset$, then it must be the case that $y=xp$ for some $p\in \mathbb{F}_k^+$ and $x\mathbb{F}_k^+\cap y\mathbb{F}_k^+=y\mathbb{F}_k^+$. In this case, Lemma \ref{lm.LCM.N.1} proves that $\alpha_x(1)\alpha_y(1)=\alpha_y(1)$. 

Otherwise, there must be an index $1\leq s\leq n$ such that $x_t=y_t$ for all $1\leq t< s$ but $x_s\neq y_s$. In this case, $\alpha_{x_s}(1)\alpha_{y_s}(1)=0$, and thus for every $a,b\in \cA$, 
\[\alpha_{x_s}(a)\alpha_{y_s}(b)=\alpha_{x_s}(a)\alpha_{x_s}(1)\alpha_{y_s}(1)\alpha_{y_s}(b)=0.\]
As a result, $\alpha_x(1)\alpha_y(1)=0$.

Therefore, the action $\alpha$ respects the right LCM and thus we have a right LCM dynamical system. The converse follows easily from  Proposition \ref{prop.equivLCM}.
\end{proof} 

In particular, the right LCM condition is easy to verify when $k=1$ and the semigroup $\mathbb{F}_k^+$ is reduced to $\mathbb{N}$. 

\begin{corollary}  An injective semigroup dynamical system $(\cA,\mathbb{N},\alpha)$ is a right LCM dynamical system if and only if $\alpha_1(\cA)$ is an ideal of $\cA$. 
\end{corollary} 

Next, we construct a non-trivial example of a dynamical system over an AF-algebra. One may refer to \cite{KenCStarByExample} for  the basic background on AF-algebras and Bratteli diagrams. 

\begin{example} Consider the AF-algebra $\cA$ with the following Bratteli diagram:

%\begin{figure}[h]
%    \centering

 
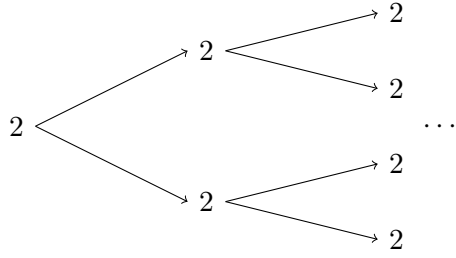
\begin{equation}
   \begin{tikzpicture}

	\draw[->] (1.25,0) -- (3.25,1);
	\draw[->] (1.25,0) -- (3.25,-1);
	
	\draw[->] (3.75,1) -- (5.75,1.5);
	\draw[->] (3.75,1) -- (5.75,0.5);
	\draw[->] (3.75,-1) -- (5.75,-0.5);
	\draw[ ->] (3.75,-1) -- (5.75,-1.5);

    %\draw[dotted] (6.25,0) -- (8.25,0);
	% gamma labels
    \node at (1,0) {$2$};
    
    \node at (3.5,1) {$2$};
    \node at (3.5,-1){$2$};
    
    \node at (6,1.5){$2$};
    \node at (6,0.5){$2$};
    \node at (6,-0.5){$2$};
    \node at (6,-1.5){$2$};
    \node at (6.6,0) {$\cdots$};

 \label{fig:Bratteli}
    \end{tikzpicture}
    \end{equation}
%\end{figure}
and notice that  $\cA$ is isomorphic to $C(X)\otimes M_2$, where $X$ is the Cantor set. 

Fix $n$ momentarily and write $0$ for the zero matrix in $\oplus_{i=1}^{2^n} M_2$. The map $\alpha_1: \oplus_{i=1}^{2^n} M_2 \to \oplus_{i=1}^{2^{n+1}} M_2$ defined by $\alpha_1(a)=a\oplus 0$ 
is a $\ast$-endomorphism of $\cA$. One can easily verify that its image, corresponding to the upper branch of the Bratteli diagram, is an ideal of $\cA$. Therefore, the dynamical system $(\cA, \mathbb{N},\alpha)$ is a right LCM dynamical system. Here, the map $\alpha_1$ exploits the self-similar property of the given Bratteli diagram. 

One can also define $\alpha_2(a)=0\oplus a$ and an $\mathbb{F}_2^+$ action by sending each generator $e_i$ to $\alpha_i$. Since the image of $\alpha_i$ are ideals and $\alpha_1(1)=1\oplus 0$ is orthogonal to $\alpha_2(1)=0\oplus 1$, $(\cA, \mathbb{F}_2^+,\alpha)$ is a right LCM dynamical system by Proposition \ref{prop.LCM.N}. We shall study dilation results for such systems in Example \ref{ex.dilation.Fk}. 
\end{example}

\subsection{Covariant representations}

To study operator algebras associated with dynamical systems, it is useful to have a  presentation in terms of a convenient notion of  covariance that encodes the dynamics. We define one that is suitable for right LCM dynamical systems.

\begin{definition}\label{df.CovariantPair} An {\em isometric covariant representation} $(\pi,V)$ of a right LCM dynamical system $(\cA,P,\alpha)$ is given by
\begin{enumerate}
\item a unital $\ast$-homomorphism  $\pi:\cA\to\bh{H}$ and
\item an isometric representation $V:P\to\bh{H}$
\end{enumerate}
such that for every $p\in P$ and $a\in \cA$, 
\[V(p) \pi(a) V(p)^* = \pi(\alpha_p(a)).\] 
\end{definition}

\begin{definition}\label{df.ContractiveCovariantPair} A {\em contractive covariant representation} $(\phi,T)$ of a right LCM dynamical system $(\cA,P,\alpha)$ is given by
\begin{enumerate}
\item a unital $\ast$-preserving linear map $\phi:\cA\to\bh{H}$ and
\item a {contractive representation} $T:P\to\bh{H}$
\end{enumerate}
such that for every $p\in P$ and $a\in \cA$, 
\[T(p) \phi(a) T(p)^* = \phi(\alpha_p(a)).\] 
\end{definition}

\begin{remark} There are several different notions one could have used to define contractive covariant representations of a semigroup dynamical system. For instance, in the study of semicrossed product algebras by abelian lattice ordered semigroups \cite{DFK2014} (see also \cite{Peters1984}), the requirement is usually \[\phi(a) T(p) = T(p) \phi(\alpha_p(a)).\]
The commutativity of the semigroup is required in this definition. For non-abelian semigroups, one may also consider the covariance condition
 \[T(p) \phi(a)  =  \phi(\alpha_p(a))T(p).\]
\end{remark} 

\begin{remark} If $(\pi, V)$ is  an isometric covariant representation, then 
$$V(p)^* \pi(a) V(p) = \pi(\alpha_{p^{-1}} (a))$$ for all
 $p\in P$ and $a\in\cA$. Indeed, 
\begin{align*}
V(p)^* \pi(a) V(p) &= V(p)^* V(p)\pi(1) V(p)^* \pi(a) V(p) \\
&= V(p)^* \pi(\alpha_p(1)) \pi(a) V(p)  \\
&= V(p)^* \pi(\alpha_p(1)a)  V(p) \\
&= V(p)^* \pi(\alpha_p(\alpha_{p}^{-1}(\alpha_p(1)a)))  V(p) \\
&= V(p)^* V(p) \pi(\alpha_{p^{-1}}(a)) V(p)^* V(p) = \pi(\alpha_{p^{-1}} (a)).
\end{align*}
However, the analogous property may fail for general contractive covariant representations because 
$T(p)^*T(p)$ may not be the identity and the map $\phi$ may not be multiplicative.
\end{remark} 

\begin{remark} If the action $\alpha$ is $\ast$-automorphic, then $\alpha_p(1)=1=V(p)V(p)^*$ for all $p\in P$. So every  isometric  covariant representation $(\pi,V)$ is in fact a unitary representation of $P$.
\end{remark} 

\begin{definition} Let $(\phi,T)$  be a contractive covariant representation of the right LCM dynamical system $(\cA,P,\alpha)$ on a Hilbert space $\cH$. An {\em isometric covariant dilation} of $(\phi,T)$  is an isometric covariant representation $(\pi, V)$ on a Hilbert space $\cK$ containing $\cH$, such that
\begin{enumerate}
\medskip \item 
$\phi(a) = P_\cH \pi(a)\Big|_\cH$\ for every  $a\in \cA$, and
\medskip \item 
$T(p) = P_\cH V(p)\Big|_\cH $ \ for every $p\in P$.
\end{enumerate}
In other words, $(\phi,T)$ is the compression of $(\pi,V)$ to the subspace $\cH$ of $\cK$.
\end{definition} 

Let $(\pi,V)$ be an isometric covariant representation of a right LCM dynamical system $(\cA, P,\alpha)$ on $\bh{K}$. Suppose $\cH\subseteq \cK$ is co-invariant for $V$ (this means that $\cH$ is invariant for all the $V(p)^*$ with $p\in P$). Then we can build a contractive covariant pair $(\phi,T)$ by compressing $(\pi,V)$ to the corner of $\cH$. That is, we define $\phi:\cA\to\bh{H}$ and $T:P\to\bh{H}$  by 

\[\phi(a):=P_\cH \pi(a)\bigg|_\cH \qquad \text{and} \qquad T(p):=P_\cH V(p) \bigg|_\cH.\]
The operators $\pi(a)$ and $V(p)$, when written as  $2\times 2$ operator matrices with respect to $\cH\oplus \cH^\perp$ have the form
\[
\pi(a) =\begin{bmatrix} \phi(a) & * \\ * & * \end{bmatrix}  \qquad \text{and} \qquad V(p) = \begin{bmatrix} T(p) & 0 \\ * & * \end{bmatrix}
\]
where the $0$ in the upper right hand corner for $V(p)$ reflects the fact that  $\cH$ is invariant for $V(p)^*$.
Therefore, keeping track of the upper left corner of the matrix, which is the only relevant one for the compression, we compute
\begin{align*}
\begin{bmatrix} \phi(\alpha_p(a)) & * \\ * & * \end{bmatrix} &= \pi(\alpha_p(a))  \\
&= V(p)\pi(a)V(p)^* \\
&= \begin{bmatrix} T(p) & 0 \\ * & * \end{bmatrix} \begin{bmatrix} \phi(a) & * \\ * & * \end{bmatrix} \begin{bmatrix} T(p)^* & * \\ 0 & * \end{bmatrix} \\
&= \begin{bmatrix} T(p)\phi(a)T(p)^* & * \\ * & * \end{bmatrix},
\end{align*}
 where we used the covariance of $(\pi,V)$ in the second line. Hence, after compressing to $\cH$, we get $T(p)\phi(a)T(p)^*=\phi(\alpha_p(a))$. Moreover, one can easily verify that $\phi$ is a unital $\ast$-preserving linear map (not necessarily multiplicative) and $T$ is a representation of $P$. Therefore, we have showed that the
 compression of  an isometric covariant representation $(\pi,V)$ to a subspace that is co-invariant for $V$ 
is a contractive covariant representation $(\phi,T)$, of which the given $(\pi,V)$ is obviously a covariant isometric dilation. 

The main  focus of this paper is the reverse process:
{\em When does a given contractive covariant representation $(\phi,T)$ of a right LCM dynamical system $(\cA, P,\alpha)$
 have an isometric covariant dilation $(\pi,V)$?}
A moment's thought reveals that to be realized as the compression of a $\ast$-homomorphic representation $\pi$, the given  $\phi$ would 
have to be a unital completely positive map to begin with.  As it turns out,  this is all we need to assume in order to guarantee there exists an isometric covariant dilation. Before we prove this, in Theorem \ref{thm.dym.dilation}, we need to develop the necessary machinery about dilations in the next section.  

\section{Naimark Dilation on Semigroup Dynamical Systems}\label{sect.Naimark}
The goal of this section is to prove a generalization of Naimark's dilation theorem valid  for semigroup dynamical systems. Throughout the section, we only require that $P$ be a left-cancellative semigroup. In particular we 
do not assume that the semigroup is right LCM or even if it is, that the dynamical system respects the LCM structure. %I may need to double check where we need the left-cancellation

 A \emph{Toeplitz kernel} for a left-cancellative semigroup $P$ is a map $K:P\times P\to\bh{H}$ such that 
 \[
 K(e,e)=I, \quad K(p,q)=K(q,p)^*, \quad \text{and} \quad K(rp,rq)=K(p,q)
 \] for all $p,q,r\in P$. A Toeplitz kernel is \emph{positive definite} if for any choice of $p_1,\cdots, p_n\in P$, the operator matrix $[K(p_i, p_j)]$ is positive. The original Naimark dilation establishes that positive definite Toeplitz kernels arise from certain isometric representations. The following generalization of Naimark's theorem is due to Popescu. 

\begin{theorem}[Theorem 3.2 of \cite{Popescu1999b}] Let $P$ be a left-cancellative unital semigroup and let $K$ be a unital kernel for $P$ on the Hilbert space $\cH$. Then $K$ is a positive definite Toeplitz  kernel if and only if there exists a (minimal) isometric representation $V:P\to\bh{K}$ on some Hilbert space $\cK\supset\cH$, such that for every $p,q\in P$,
\[ K(p,q)=P_\cH V(p)^* V(q) \Big|_\cH \quad\text{ and }  \quad \overline{\operatorname{span}} \{V(p)h: p\in P, \ h\in \cH\} = \cK.\]
The minimal dilation is unique up to unitary equivalence. 
\end{theorem}

\subsection{Kernel systems}

We first define and study kernel systems associated to semigroup dynamical systems.  These  prove to be an essential tool in dilation. The kernel system defined here can be seen as a strengthened version of positive definite Toeplitz kernels on semigroups, as in the traditional Naimark dilation.

When  $(\cA,P,\alpha)$ is a semigroup dynamical system, we define $\cA_{p,q}=\{\alpha_p(1) a \alpha_q(1): a\in \cA\}$ to be the `off diagonal'  corner of the $\mathrm{C}^*$-algebra $\cA$ corresponding to $p,q \in P$. One can easily verify that if $a\in \cA_{p,q}$, then $a^*\in \cA_{q,p}$. 

\begin{remark} Notice that when $(\cA,P,\alpha)$ is a right LCM dynamical system, then $\cA_{p,q}=\alpha_p(\cA)\cap \alpha_q(\cA)=\alpha_p(\cA)\alpha_q(\cA)$. This is because in this case $\alpha_p(\cA),\alpha_q(\cA)$ are ideals so that for any $a\in \cA$, $\alpha_p(1)a\alpha_q(1)\in \alpha_p(\cA)\cap \alpha_q(\cA)$. Conversely, if $a\in \alpha_p(\cA)\cap \alpha_q(\cA)$, then $a=\alpha_p(1)a\alpha_q(1)\in \cA_{p,q}$. 
\end{remark} 
\begin{lemma} \label{lem:oldremark4.4}  Suppose $(\cA,P,\alpha)$ is a semigroup dynamical system. Then  $\alpha_r(\cA_{p,q})\subseteq \cA_{rp,rq}$ for every $p,q, r\in P$. Moreover, $\alpha_r(\cA_{p,q})= \cA_{rp,rq}$ for every $p,q,r\in P$ if and only if 
  $\alpha_r(\cA_{e,e})$ is hereditary for every $r\in P$.
\end{lemma}
\begin{proof}
For the first assertion, take $a=\alpha_p(1) b \alpha_q(1)\in\cA_{p,q}$ and notice that
\[\alpha_r(a)=\alpha_{rp}(1) \alpha_r(b) \alpha_{rq}(1)\in \cA_{rp,rq}.\]
 If equality holds for every $r,p,q$, then, then in particular, $\alpha_r(\cA)=\alpha_r(\cA_{e,e})=\cA_{r,r} =\alpha_r(1)\cA\alpha_r(1)$, hence 
$\alpha_r(\cA)$ is hereditary by \cite[Lemma 4.1]{Murphy1996}. Conversely, assume $\alpha_r(\cA)$ is hereditary for every $r$; then
\[\alpha_r(\cA_{p,q}) = \alpha_{rp}(1) \alpha_r(\cA) \alpha_{rq}(1) 
= \alpha_{rp}(1) \alpha_r(1) \cA \alpha_r(1) \alpha_{rq}(1) = \cA_{rp,rq},\]
and equality holds.
\end{proof}
We  use $e$ for the identity element in the semigroup, $1$ for the unit of the $\mathrm{C}^*$-algebra $\cA$, and
$I$ for the identity operator on a Hilbert space. 

\begin{definition}\label{df.kernel} 
Let $(\cA,P,\alpha)$ be a semigroup dynamical system and define 
\[\Lambda_{(\cA,P,\alpha)}=\{(p,a,q): p,q\in P, a\in \cA_{p,q}\}.\]
A \emph{kernel system} for $(\cA,P,\alpha)$ on a Hilbert space $\cH$ is  a map $K:\Lambda_{(\cA,P,\alpha)}\to\bh{H}$. We say that $K$ is 
\begin{enumerate}
\item \emph{unital} if $K(e,1,e)=I$,
\item \emph{Hermitian} if $K(p,a,q)^*=K(q,a^*,p)$,
\item \emph{Toeplitz} if $K(p, a, q)=K(rp, \alpha_r(a), rq)$ (notice that $(rp, \alpha_r(a), rq)  \in \Lambda_{(\cA,P,\alpha)}$ by Lemma~\ref{lem:oldremark4.4}),
\item \emph{linear} if $K(p, a,q)+\lambda K(p,b,q)=K(p, a+\lambda b,q)$ for every $a,b\in \cA_{p,q}$ and $\lambda\in\mathbb{C}$,
\item \emph{positive} if for every choice $p_1,\cdots, p_n\in P$ and $a_1,\cdots, a_n\in \cA$ with $a_i\in \cA \alpha_{p_i}(1)$, we have 
\[ [K(p_i, a_i^* a_j, p_j)]\geq 0,\]
\item \emph{bounded} if for any $p_1,\cdots, p_n\in P$ and $a_1,\cdots, a_n\in \cA$ with $a_i\in \cA \alpha_{p_i}(1)$, and for every $a\in A$, 
 \[ [K(p_i, a_i^* a^* a a_j, p_j)]\leq \|a\|^2 [K(p_i, a_i^* a_j, p_j)],\]
\item \emph{continuous} if for every $p,q\in P$, and every sequence $(a_n)$ in $\cA_{p,q}$ that converges to $a\in \cA_{p,q}$, then $K(p,a_n,q)$ also converges to $K(p,a,q)$ (in norm). 
\end{enumerate}
\end{definition}

\begin{lemma}\label{lm.bounded} Every linear positive kernel system $K$ is bounded.
\end{lemma} 

\begin{proof} Let $p_1,\cdots, p_n\in P$, and for each $i = 1,2, \ldots n$ choose  $a_i\in \cA \alpha_{p_i}(1)$. For each $a\in\cA$, define $b_i =  (\|a\|^2I-a^*a)^{1/2} a_i$, which is again in $\cA \alpha_{p_i}(1)$. The positivity of $K$ implies that the operator matrix  
\[[K(p_i, b_i^* b_j, p_j)] = [K(p_i, a_i^* (\|a\|^2I - a^* a) a_j, p_j)]\]
is positive definite.
The linearity of $K$ implies that 
\[[K(p_i, a_i^* (\|a\|^2I - a^* a) a_j, p_j)] = \|a\|^2 [K(p_i, a_i^* a_j, p_j)] - [K(p_i, a_i^* a^* a a_j, p_j)] \geq 0,\]
which gives  the desired inequality, proving  that $K$ is bounded.
\end{proof} 

\begin{lemma}\label{lm.cts} If $K$ is a unital, Hermitian, Toeplitz, linear, positive kernel system, then  
$\|K(p,a,q)\|\leq \|a\|$, for all $p,q\in P$ and $a\in \cA_{p,q}$, and thus $K$ is continuous. 
\end{lemma}

\begin{proof} Fix $p,q\in P$ and $a\in \cA_{p,q}$. Let $a_1=\alpha_p(1)$ and $a_2=a$, the positivity assumption ensures that 
\[\begin{bmatrix}
K(p,\alpha_p(1),p) & K(p,a,q) \\
K(q, a^*, p) & K(q, a^*a, q) 
\end{bmatrix} \geq 0. \]

By Lemma \ref{lm.bounded}, $K$ is also bounded so that 
\[K(q,a^*a,q)\leq \|a\|^2 K(q,\alpha_q(1),q)=\|a\|^2 K(e,I,e)=\|a\|^2 I.\]
Also note that $K(p,\alpha_p(1),p)=K(e,I,e)=I$ because $K$ is Toeplitz. Therefore,
\[\begin{bmatrix}
K(p,\alpha_p(1),p) & K(p,a,q) \\
K(q, a^*, p) & \|a\|^2 I 
\end{bmatrix}
\geq 
\begin{bmatrix}
K(p,\alpha_p(1),p) & K(p,a,q) \\
K(q, a^*, p) & K(q, a^*a, q) 
\end{bmatrix} \geq 0 \]
This can only happen when $K(q,a^*,p)K(p,a,q)=K(p,a,q)^* K(p,a,q)\leq \|a\|^2 I$ (by using for example \cite[Lemma 14.13]{NestAlgebra}). Therefore, $\|K(p,a,q)\|\leq \|a\|$ as desired. \end{proof} 

It is often easier to verify the properties of a kernel system on a dense subset of $\cA$; the following proposition shows that it is also enough.

\begin{proposition} Suppose $(\cA,P,\alpha)$ is a semigroup dynamical system and let $\cA_0$ be a unital dense $\ast$-subalgebra of $\cA$ that is invariant under $\alpha$. Then every unital, Hermitian, Toeplitz, linear, positive kernel system $K_0$ on $(\cA_0,\alpha,P)$ can be extended to a unital, Hermitian, Toeplitz, linear, positive kernel system $K$ on $(\cA,\alpha,P)$. 
\end{proposition} 

\begin{proof} If $a_n\in\cA_{p,q}$ converges to $a\in \cA_{p,q}$, linearity and Lemma \ref{lm.cts} give
\[\|K_0(p,a_n,q)-K_0(p,a_m,q)\|=\|K_0(p,a_n-a_m,q)\|\leq \|a_n-a_m\|.\]
Therefore, $K_0(p,a_n,q)$ is a Cauchy sequence that converges in $\bh{H}$. Define $K(p,a,q) :=\lim_{n\to\infty} K_0(p,a_n,q)$. One can easily verify that $K$ is unital, Hermitian, Toeplitz (since the $\alpha_p$ are continuous $\ast$-endomorphisms), and linear. For positivity, pick $p_1,\cdots,p_n\in P$ and $a_1,\cdots,a_n\in\cA$ with $a_i\in \cA \alpha_{p_i}(1)$. For each $i$, pick a sequence $a_{i,n}$ in $\cA\alpha_{p_i}(1)$ that converges to $a_i$. Then $K_0(p_i, a_{i,n}^* a_{j,n}, p_j)$ converges to $K(p_i, a_i^* a_j, p_j)$, and since the matrix 
$[K_0(p_i, a_{i,n}^* a_{j,n}, p_j)]$ is positive  definite 
so is  $[K(p_i, a_i^* a_j, p_j)] $. This proves $K$ is positive. 
\end{proof}

\subsection{Naimark dilation for kernel systems}
Naimark dilation is a powerful tool in the study of dilation theory on semigroups because it explicitly constructs isometric representations of the semigroup via positive definite kernels \cite{Popescu1999b, BLi2014, BLi2019}. With the additional $\mathrm{C}^*$-algebra $\cA$ in our definition of the kernel system, we aim to establish an extended version of Naimark dilation theorem that explicitly constructs both a representation of the $\mathrm{C}^*$-algebra $\cA$ and an isometric representation of the semigroup $P$. 

The usual way of constructing a kernel is by compressing an isometric covariant representation to a subspace. The next proposition shows that this also works for kernel systems. We emphasize that only the weaker covariance condition is involved in the process.

\begin{proposition} 
Let $(\pi,V)$ be an isometric  representation of a semigroup dynamical system $(\cA, \alpha, P)$ on a Hilbert space $\cK$ that satisfies the (weak) covariance condition $\pi(\alpha_p(a))V(p)=V(p)\pi(a)$. Suppose $\cH$ is a subspace of $\cK$ and define $K:\Lambda_{(\cA,P,\alpha)}\to\bh{H}$ by 
\[K(p,a,q)=P_\cH V(p)^* \pi(a) V(q) \bigg|_\cH.\]
Then $K$ satisfies conditions (1) to (6) in  Definition \ref{df.kernel}.
\end{proposition} 

\begin{proof} For condition (1), notice that  $K(e, 1, e)=P_\cH V(e)^* \pi(1) V(e)\bigg|_\cH$  is equal to  $I$ because $V$ is an isometric representation with $V(e)=I$ and  $\pi(1)=I$.

For condition (2), simply compute  \[K(p,a,q)^* = P_\cH (V(p)^* \pi(a) V(q))^* \bigg|_\cH = P_\cH V(q)^* \pi(a^*) V(p) \bigg|_\cH=K(q, a^*, p)\]

For condition (3), 
\[V(rp)^* \pi(\alpha_r(a)) V(rq) = V(p)^* V(r)^* \pi(\alpha_r(a)) V(r) V(q)=V(p)^* V(r)^*  V(r) \pi(a) V(q)\]
cancelling $ V(r)^*  V(r)= I$ and compressing down to $\cH$ gives $K(rp,\alpha_r(a),rq)= K(p,a,q)$. 

For condition (4), 
\[V(p)^*\pi(a+\lambda b) V(q)=V(p)^*\pi(a)V(q) + \lambda V(p)^* \pi(b) V(q),\]
again, compressing down to $\cH$ gives $K(p,a+\lambda b,q)=K(p,a,q)+\lambda K(p,b,q)$.

For condition (5), fix $p_1,\cdots, p_n\in P$ and $a_1,\cdots, a_n\in \cA$ with $a_i\in \cA\alpha_{p_i}(1)$, and define an $1\times n$ operator matrix $R$ to be 
\[R:=\begin{bmatrix} \pi(a_1) V(p_1) & \cdots & \pi(a_n) V(p_n) \end{bmatrix}. \]
Then,
\[R^* R=[V(p_i)^* \pi(a_i^* a_j) V(p_j)]\geq 0\]
Compressing down to $\cH$ proves that $[K(p_i, a_i^* a_j, p_j)]\geq 0$.

Finally, condition (6) follows from linearity, positivity, and Lemma \ref{lm.bounded}.
\end{proof} 

The extended Naimark dilation theorem tackles the converse. 

\begin{theorem}\label{thm.kernel.dilation} Let $K:\Lambda_{(\cA,P,\alpha)}\to\bh{H}$ be a kernel system that satisfies conditions (1) to (5) of  Definition \ref{df.kernel}. Then there exists a Hilbert space $\cK \supseteq \cH$,  an isometric representation $V:P\to\bh{K}$ and a $\ast$-homomorphism $\pi:\cA\to\bh{K}$ satisfying 
\[V(p) \pi(a) = \pi(\alpha_p(a)) V(p) \qquad \text{ for } p,q \in P, \  a\in \cA,\]
%for  $p,q\in P$ and $a\in \cA_q$
%with $\cH\subseteq \cK$, 
such that  \[K(p,a,q)=P_\cH V(p)^* \pi(a) V(q) \bigg|_\cH \qquad \text{ for } (p,a,q)\in\Lambda_{(\cA,P,\alpha)}.\]

Moreoever, $(\pi, V)$ can be taken to be minimal in the sense that 
\[\cK = \overline{\lspan}\{\pi(a) V(p)h: p\in P, a\in \cA\alpha_{p}(1), h\in \cH\subset\cK\}.\]
Such a minimal dilation $(\pi, V)$ is unique up to unitary equivalence.
\end{theorem}

\begin{proof} 
Consider the discrete space $\cX := \{(p,a): p\in P, a\in\cA \alpha_p(1)\}$,  and 
write $\delta_{p,a} \otimes h$ for the (elementary) function that equals $h$ at $(p,a)$ and is zero everywhere else.  Let
\[\cK_0: =C_c(\cX; \cH) =\lspan\{\delta_{p,a}\otimes h: p\in P, a\in\cA \alpha_p(1), h\in\cH\}.
\]

If $(p,a)$ and $(q,b)$ are in $\cX$, then $b^* a \in \alpha_q(1) \cA \alpha_p(1)=\cA_{q,p}$, so we may  define  
\[\langle\delta_{p,a}\otimes h, \delta_{q,b}\otimes k\rangle := \langle K(q,b^*a, p) h, k\rangle\]
which can be extended uniquely to a sesquilinear form on $\cK_0$. This sesquilinear form is positive definite because for every $k=\sum_{i=1}^n \delta_{p_i, a_i}\otimes h_i$,
\begin{align*}
\langle k, k\rangle &= \sum_{i=1}^n \sum_{j=1}^n \langle K(p_j, a_j^* a_i, p_i) h_i, h_j\rangle \\
&= \left\langle [K(p_j, a_j^* a_i, p_i)] \begin{bmatrix} h_1 \\ \vdots \\ h_n\end{bmatrix}, \begin{bmatrix} h_1 \\ \vdots \\ h_n\end{bmatrix} \right\rangle\geq 0,
\end{align*}
where the last inner product is in $\cH^n$.
Define $\|k\|=\langle k,k\rangle^{1/2}$. The sesquilinear form $\langle\cdot, \cdot\rangle$ is a pre-inner product that satisfies the Cauchy-Schwarz inequality $|\langle k,h\rangle|^2\leq \|k\| \|h\|$. If we now let 
\[\cN:=\{k\in \cK_0: \|k\|=0\}=\{k\in\cK_0: \langle k,h\rangle =0\mbox{ for all }h\in\cK_0\},\] %{thm.kernel.dilation}
we see that $\cN$ is a linear subspace of $\cK_0$ and that $\langle \cdot,\cdot\rangle$  becomes an inner product on the quotient $\cK_0/\cN$.  The completion of $\cK_0/\cN$ with respect to the associated norm will be denoted 
$\cK$. On the image of ${\cK_0/ \cN}$ in this completion 
the inner product  is given by
\[\langle k+\cN, h+\cN\rangle = \langle k, h\rangle\]
and the norm is $\|k+\cN\|=\langle k+\cN, k+\cN\rangle^{1/2}$. 
%For simplicity, we shall write $k\in\cK$ instead of $k+\cN$. This seems to say that the completion is a 
Since
\[
\langle \delta_{e,1}\otimes h, \delta_{e,1}\otimes k\rangle = \langle K(e,1,e)h, k\rangle =\langle h, k\rangle,
\]
the Hilbert space $\cH$ naturally embeds in $\cK$ via $h\to \delta_{e,1}\otimes h$.

Notice that if $b\in \cA\alpha_q(1)$,  then $\alpha_p(b)\in \cA\alpha_{pq}(1)$, so that $(pq,\alpha_p(b)) \in \cX$ and we may define 
 $V(p):\cK_0\to\cK_0$ by \[V(p)\delta_{q,b}\otimes h = \delta_{pq,\alpha_p(b)}\otimes h,\] 
 for $p,q\in P$ and $b\in \cA\alpha_q(1)$.  The operator $V(p)$ is isometric because for every linear combination 
 $k=\sum_{i=1}^n \delta_{p_i, a_i}\otimes h_i$, 
\begin{align*}
\|V(p)k\|^2 &= \|\sum_{i=1}^n \delta_{pp_i, \alpha_p(a_i)}\otimes h_i\|^2 \\
&= \sum_{i=1}^n \sum_{j=1}^n \langle K(pp_j, \alpha_p(a_j)^* \alpha_p(a_i), pp_i) h_i, h_j\rangle \\
&= \sum_{i=1}^n \sum_{j=1}^n \langle K(pp_j, \alpha_p(a_j^* a_i), pp_i) h_i, h_j\rangle \\
&= \sum_{i=1}^n \sum_{j=1}^n \langle K(p_j, a_j^* a_i, p_i) h_i, h_j\rangle = \|k\|^2;
\end{align*}
where we have used the Toeplitz condition in the last line.

Noticing now that if $b\in \cA\alpha_q(1)$ and $a\in \cA$, then $ab\in\cA\alpha_q(1)$, so that $(q,ab) \in \cX$, we define a map
 $\pi(a): \cK_0\to\cK_0$ by first  letting
\[\pi(a) (\delta_{q,b}\otimes h)=\delta_{q,ab}\otimes h\] 
and then extending it by linearity.
In order to show that $\pi(a)$ is bounded, we take $k=\sum_{i=1}^n \delta_{p_i, a_i}\otimes h_i$ and  compute
\begin{align*}
\|\pi(a)k\|^2 &= \|\sum_{i=1}^n \delta_{p_i, a a_i}\otimes h_i\|^2 \\
&= \sum_{i=1}^n \sum_{j=1}^n \langle K(p_j, a_j^* a^*a a_i, p_i) h_i, h_j\rangle \\
&\leq \|a\|^2 \sum_{i=1}^n \sum_{j=1}^n \langle K(p_j, a_j^* a_i, p_i) h_i, h_j\rangle \\
&= \|a\|^2 \|k\|^2,
\end{align*}
where the inequality holds because the kernel $K$ is assumed to be bounded.

Since $V(p)$ and $\pi(a)$ are both bounded on $\cK_0$ and leave $\cN$  invariant, they determine bounded linear operators on $\cK$, which we also denote by $V(p)$ and $ \pi(a)$. We now verify that $(\pi, V)$ has the desired properties.

\noindent\textbf{Claim 1}. $V:P\to\bh{K}$ is an isometric representation. 

We have proved each $V(p)$ is an isometry. Let $p,q\in P$; then
\begin{align*}
V(p)V(q)\delta_{r,a}\otimes h &= \delta_{pqr,\alpha_{p}(\alpha_q(a))}\otimes h \\
&= \delta_{pqr,\alpha_{pq}(a)}\otimes h \\
&=V(pq)\delta_{r,a}\otimes h.
\end{align*}
Therefore, $V(pq)=V(p)V(q)$. That $V(e)=I$ is obvious from the definition.

\noindent \textbf{Claim 2}. $\pi:\cA\to\bh{K}$ is a unital $\ast$-homomorphism. 

It is clear that $\pi(1)=I$.
To prove $\pi$ is a homomorphism, pick any $a,b\in\cA$. It suffices to show that $\pi(a)\pi(b) = \pi(ab)$ on a set that has dense linear span, so choose $(r,c) \in \cX$ and $h\in \cH$; then
\[\pi(a)\pi(b)\delta_{r,c}\otimes h=\delta_{r,abc}\otimes h=\pi(ab)\delta_{r,c} \otimes h.\]
 It remains to show that $\pi(a^*)=\pi(a)^*$, and for this it suffices to prove that for every pair of elementary functions $\delta_{p,b}\otimes h$ and $  \delta_{q,c}\otimes k$ in $\cK$, we have
\[\langle \pi(a^*)\delta_{p,b}\otimes h, \delta_{q,c}\otimes k\rangle = \langle \pi(a)^*\delta_{p,b}\otimes h, \delta_{q,c}\otimes k\rangle\]
To verify this we compute
\begin{align*}
\langle \pi(a^*)\delta_{p,b}\otimes h, \delta_{q,c}\otimes k\rangle &=
\langle \delta_{p,a^* b}\otimes h, \delta_{q,c}\otimes k\rangle \\
&= \langle K(q,c^* a^* b,p)h, k\rangle \\
&= \langle \delta_{p,b}\otimes h, \delta_{q,ac}\otimes k\rangle \\
&= \langle \delta_{p,b}\otimes h, \pi(a)\delta_{q,c}\otimes k\rangle.
\end{align*}

\noindent \textbf{Claim 3}.
$K(p,a,q)=P_\cH V(p)^* \pi(a) V(q) \bigg|_\cH$.

For any $h,k\in \cH$, which embed in $\cK$ as $\delta_{e,1}\otimes h, \delta_{e,1}\otimes k$,
\begin{align*}
\langle V(p)^* \pi(a) V(q) \delta_{e,1}\otimes h, \delta_{e,1}\otimes k\rangle &= \langle \pi(a) V(q) \delta_{e,1}\otimes h, V(p) \delta_{e,1}\otimes k\rangle  \\
&= \langle \delta_{q,a}\otimes h, \delta_{p, \alpha_p(1}\otimes k\rangle \\
&= \langle K(p, \alpha_p(1) a, q)h, k\rangle\\
&= \langle K(p, a, q)h, k\rangle
\end{align*}
Here, we use the fact that $a\in \cA_{p,q}=\alpha_p(1)\cA\alpha_q(1)$ and thus $\alpha_p(1) a=a.$

\noindent\textbf{Claim 4}.  
$V(p) \pi(a) = \pi(\alpha_p(a)) V(p)$.

Take any elementary function $\delta_{b,r}\otimes h\in \cK$. Then
\begin{align*}
V(p)\pi(a) \delta_{b,r}\otimes h &= \delta_{\alpha_p(ab), pr}\otimes h \\
&= \delta_{\alpha_p(a)\alpha_p(b), pr}\otimes h \\
&= \pi(\alpha_p(a)) V(p) \delta_{b,r}\otimes h. 
\end{align*}
Since $V(p)\pi(a)$ and $\pi(\alpha_p(a)) V(p)$ are bounded linear operators, the equality also holds for every vector in $\cK$. 

\noindent\textbf{Claim 5}.  $(\pi,V)$ is minimal.

Since $\pi(a)V(p)h=\delta_{a,p}\otimes h$ we have
\begin{multline*}\lspan\{\pi(a)V(p)h: p\in P, a\in \cA \alpha_p(1), h\in \cH\} = \\ \lspan\{\delta_{a,p}\otimes h: p\in P, a\in \cA \alpha_p(1), h\in \cH\},
\end{multline*}
which is clearly dense in $\cK$.

\noindent\textbf{Claim 6}.  $(\pi,V)$ is unique up to unitary equivalence.

Suppose $\widetilde{\pi}$ and $\widetilde{V}$ is another minimal dilation on a Hilbert space $\widetilde{\cK}$ and consider
\[x=\sum_{i\in F} \widetilde{\pi}(a_i) \widetilde{V}(p_i) h_i, \qquad  y=\sum_{i\in F} \widetilde{\pi}(a_i) \widetilde{V}(p_i) k_i.\]
Their $\widetilde{\cK}$ inner product is given by
\[
\langle x,y\rangle = \sum_{i,j\in F} \langle \widetilde{V}(p_i)^* \widetilde{\pi}(a_i^* a_j) \widetilde{V}(p_j) h_j, k_i\rangle.\]
Since $K(p,a,q)=P_\cH V(p)^* \pi(a) V(q)\bigg|_\cH$ for all $(p,a,q)\in \Lambda_{(\cA,P,\alpha)}$, 
\[\langle \widetilde{V}(p_i)^* \widetilde{\pi}(a_i^* a_j) \widetilde{V}(p_j) h_j, k_i\rangle = \langle K(p_i, a_i^* a_j, p_j) h_j, k_i\rangle.\]
Therefore, 
\[\langle x,y\rangle = \sum_{i,j\in F}  \langle K(p_i, a_i^* a_j, p_j) h_j, k_i\rangle.\]
A similar argument can be applied to the two vectors 
\[Ux=\sum_{i\in F} \pi(a_i) V(p_i) h_i, \qquad Uy=\sum_{i\in F} \pi(a_i) V(p_i) k_i,\]
proving that $\langle Ux, Uy\rangle = \langle x,y\rangle$ for all $x,y\in \lspan\{\widetilde{\pi}(a)\widetilde{V}(p)h: p\in P, a\in \cA \alpha_p(1),h\in\cH\}$. Therefore, $U$ can be extended to a unitary on its closure $\widetilde{\cK}$. 
It is routine to check that $\widetilde{\pi}(a)=U^* \pi(a) U$ and $\widetilde{V}(p)=U^* V(p) U$ for all $a\in \cA$ and $p\in P$. This establishes that the minimal dilation $(\pi, V)$ is unique up to unitary equivalence and completes the proof.
\end{proof} 

Notice that Theorem \ref{thm.kernel.dilation}  only involves the weaker covariance condition $V(p)\pi(a)=\pi(\alpha_p(a))V(p)$, 
and that, in principle, the minimal dilation $(\pi,V)$  may not satisfy the stronger covariance condition $V(p)\pi(a)V(p)^*=\pi(\alpha_p(a))$. A major difficulty is that the construction is not explicit about the vectors $V(p)^*\delta_{q,b}\otimes h$. However, in the particular case when
$V(p)V(p)^*=\pi(\alpha_p(1))$ holds for all $p\in P$, then we can recover the covariance condition, as the following proposition shows. 
\begin{proposition}\label{prop.cov.equiv} Suppose $\pi:\cA\to\bh{K}$ is a unital $\ast$-homomorphism and $V:P\to\bh{K}$ is an isometric representation. Then the following are equivalent:
\begin{enumerate}
\item For every $p\in P$ and $a\in \cA$, \[V(p)\pi(a)V(p)^*=\pi(\alpha_p(a)).\]
\item For every $p\in P$ and $a\in \cA$, \[V(p)\pi(a)=\pi(\alpha_p(a))V(p).\]
and $V(p)V(p)^*=\pi(\alpha_p(1))$. 
\end{enumerate}
\end{proposition} 

\begin{proof} The conditions in (2) follow from (1) by multiplying $V(p)$ on the right hand side and setting $a=1$. To see the converse, assume (2) holds and multiply $V(p)^*$ on the right on both sides,
\begin{align*}
V(p)\pi(a)V(p)^* &= \pi(\alpha_p(a))V(p)V(p)^* \\
&= \pi(\alpha_p(a)) \pi(\alpha_p(1)) \\
&= \pi(\alpha_p(a)). \qedhere
\end{align*}

\end{proof} 

\begin{remark} The connection between the covariance relations $V(p)\pi(a)=\pi(\alpha_p(a))V(p)$ and $V(p)\pi(a)V(p)^*=\pi(\alpha_p(a))$ is further explored in \cite{Li2020env}, where it is shown that when the semigroup is abelian, the former relation can be dilated to the latter one. 
\end{remark} 

\begin{remark} Even though our main focus is on right LCM dynamical systems, our version of generalized Naimark dilation can be applied to a much wider range of semigroup dynamical systems. 

Recall the semigroup dynamical system arising from the Cuntz algebra $\cO_k$. Let $S_1, \cdots, S_k$ be isometries that generate $\cO_k$, and for any word $w=w_1w_2\cdots w_n\in\mathbb{F}_k^+$, let $|w|=n$ be the length of the word $w$ and $S_w=S_{w_1}S_{w_2}\cdots S_{w_n}$. Recall that the core $\cA=\overline{\lspan}\{S_\mu S_\nu^*: |\mu|=|\nu|\}$ of $\cO_k$  is isomorphic to the UHF algebra of type $k^\infty$. One can define the action $\alpha_1(x)=e_{11}\otimes x$ on the UHF core, which is the same as $\alpha_1(x)=S_1 x S_1^*$. We have a semigroup dynamical system $(\cA,\mathbb{N},\alpha)$, where the $\mathbb{N}$-action $\alpha$ is induced by the single $\ast$-endomorphism $\alpha_1$ via $\alpha_n:=\alpha_1^n$. For every $m,n \in \mathbb{N}$, one can explicitly compute the corner $\cA_{n,m}=\alpha_n(1)\cA \alpha_m(1)$ as 
\begin{align*}
\cA_{n,m}&=\overline{\lspan}\{S_\mu S_\nu^*: |\mu|=|\nu|, \mu=\underbrace{1\cdots 1}_{n}\mu', \nu=\underbrace{1\cdots 1}_{m}\nu'\} \\
&=\overline{\lspan}\{S_1^n S_{\mu'} S_{\nu'}^* S_1^{*m}: |\mu'|+n=|\nu'|+m\} \\
\end{align*}

Let $T_1,\cdots,T_k$ be contractions in $\bh{H}$ such that $\sum_{i=1}^k T_i T_i^*=I$. For any $w\in\mathbb{F}_k^+$, we denote $T_w$ in a similar way as $S_w$. One can define a kernel system $K$ by 
\[K(n, S_1^n S_{\mu'} S_{\nu'}^* S_1^{*m}, m)=T_{\mu'} T_{\nu'}^*.\]
This kernel is from the compression of the identity map of $\cA$ and the isometric representation $V(n)=S_1^n$. Therefore, it does satisfy all the conditions in  Definition \ref{df.kernel}. This is not surprising due to the well-known  dilation theorem of Popescu, \cite[Theorem~3.2]{Popescu1999b}, whereby the $T_i$ can be dilated to Cuntz isometries.  
\end{remark} 

%\begin{example} Let $n\geq 2$ be any natural number. It is observed that the Cuntz-algebra $\cO_n$ can be realized as a cross product $\cA\rtimes \mathbb{N}$, where $\cA$ is a UHF algebra of type $n^\infty$ \cite{Stacey1993} (see also \cite{KenCStarByExample}). 
%\end{example} 

\section{Covariant Pairs and Dilation}\label{sect.dilation}

We turn our attention back to right LCM dynamical systems. The goal of this section is to prove the following result, which can be seen as an equivariant version of Stinespring's theorem for right LCM dynamical systems. 

\begin{theorem}\label{thm.dym.dilation} Let $(\phi,T)$ be a contractive covariant representation of a right LCM dynamical system $(\cA,P,\alpha)$. Then there exists an isometric covariant dilation $(\pi,V)$ of $(\phi,T)$ if and only if $\phi$ is a unital completely positive map. Moreover, the dilation $(\pi,V)$ can be taken to be minimal and the minimal dilation is unique up to unitary equivalence.
\end{theorem} 

\begin{remark} Consider the trivial case where $P=\{e\}$. Then Theorem \ref{thm.dym.dilation} is simply  Stinespring's dilation theorem which states that every unital completely positive map can be dilated to a $\ast$-homomorphism. 

For nontrivial $P$, Stinespring's theorem only gives a $\ast$-homomorphic dilation $\pi$ of the unital completely positive map $\phi$. Surprisingly, it also allows us to construct an isometric dilation $V$ alongside, such that $(\pi, V)$ satisfies the covariance condition. 
\end{remark} 

The proof is divided into two parts. First  we construct a kernel system $K$ from a contractive covariant representation $(\phi,T)$
and show, in a series of lemmas and propositions, that the kernel system $K$ enjoys all the nice properties from Definition \ref{df.kernel} whenever $\phi$ is unital completely positive. This allows us to invoke Theorem \ref{thm.kernel.dilation} to obtain a $\ast$-homomorphism $\pi$ of the $\mathrm{C}^*$-algebra and an isometric representation $V$ of the semigroup. Then the proof is concluded by verifying that $(\pi,V)$ satisfies the covariance condition from Definition \ref{df.CovariantPair}.  

%\subsection{Kernel Systems from Contractive Covariant Representations}

We begin by constructing  kernel systems from  contractive covariant representations.
 \begin{lemma}\label{lem:welldefined} Let $(\phi,T)$ be a contractive covariant representation of the injective right LCM dynamical system $(\cA, P, \alpha)$. Suppose that $rP=sP=pP\cap qP$.
Then
\[T(p^{-1}r) \phi(\alpha_r^{-1}(a)) T(q^{-1}r)^* = T(p^{-1}s) \phi(\alpha_s^{-1}(a)) T(q^{-1}s)^* \]
\end{lemma}

\begin{proof} If $rP=sP=pP\cap qP$, then $r=su$ for some invertible element $u\in P$, so we have \begin{align*}
T(p^{-1}r) \phi(\alpha_r^{-1}(a)) T(q^{-1}r)^* &= T(p^{-1} s) T(u) \phi(\alpha_u^{-1}(\alpha_s^{-1}(a))) T(u)^* T(q^{-1} s)^* \\
&= T(p^{-1}s) \phi(\alpha_s^{-1}(a)) T(q^{-1}s)^*. \qedhere
\end{align*}

\end{proof}

\begin{definition}\label{df.K} The \emph{kernel system $K$ associated with a contractive covariant representation $(\phi,T)$} is the function on 
$\Lambda_{(\cA,P,\alpha)}$ defined by  
\[
K(p,a,q):=\begin{cases}T(p^{-1}r) \phi(\alpha_r^{-1}(a)) T(q^{-1}r)^* & \text{ if } rP=pP\cap qP \neq \emptyset;\\
0& \text{ if } pP\cap qP = \emptyset.
\end{cases}  
\]
 Notice  $K$ is well-defined because of Lemma~\ref{lem:welldefined}.
\end{definition}

\begin{lemma}\label{lm.Kdef2} Suppose that $a\in \cA_s$ for some $s\in pP\cap qP$. Then 
\[K(p,a,q)=T(p^{-1}s) \phi(\alpha_s^{-1}(a)) T(q^{-1}s)^*.\]
\end{lemma} 

\begin{proof} Let $pP\cap qP=rP$;  since $s\in rP$ we have $s=rt$ for some $t\in P$. We have,

\begin{align*}
 \tcb{T(p^{-1}s) \phi(\alpha_s^{-1}(a)) T(q^{-1}s)^* }
&= T(p^{-1}r) T(t) \phi(\alpha_s^{-1}(a)) T(t)^* T(q^{-1}r)^* \\
&= T(p^{-1}r) \phi(\alpha_t(\alpha_s^{-1}(a)))  T(q^{-1}r)^* \\
&= T(p^{-1}r) \phi(\alpha_r^{-1}(a))  T(q^{-1}r)^* \\
&= K(p,a,q). \qedhere
\end{align*}
\end{proof} 
\begin{lemma}\label{lm.Kpullout} Let $p,q\in P$ and $r\in qP$. Let $a\in \cA_p\cA_r\subseteq \cA_p\cA_q$. Then 
\[K(p,a,q)=K(p,a,r)T(q^{-1}r)^*.\]
\end{lemma} 

\begin{proof} In the case when $pP\cap rP=\emptyset$, $a\in \cA_p\cA_r=\{0\}$ and thus $a=0$, and either side becomes $0$. 

Assume otherwise that $pP\cap rP=sP\subset pP\cap qP$. Then $s\in pP\cap qP$, and thus by Lemma \ref{lm.Kdef2},
\begin{align*}
K(p,a,q) &= T(p^{-1} s) \phi(\alpha_s^{-1}(a)) T(q^{-1}s)^* \\
&= T(p^{-1} s) \phi(\alpha_s^{-1}(a)) T(r^{-1}s)^* T(q^{-1} r)^* \\
&= K(p,a,r)T(q^{-1}r)^*. \qedhere
\end{align*}
\end{proof} 

\begin{proposition}\label{prop.Kbasic} Let $K$ be the kernel system associated with $(\phi,T)$. Then $K$ is unital, Hermitian, Toeplitz, and linear. 
\end{proposition} 

\begin{proof} The kernel $K$ is unital because $K(e,1,e)=T(e) \phi(1) T(e)^*=I$. To see $K$ is Hermitian, consider first the case when $pP\cap qP=\emptyset$; then $K(p,a,q)^*=0=K(q,a^*,p)$. Next consider the case $pP\cap qP=rP$ and use the fact that $\phi$ and $\alpha_r$ are $\ast$-maps:
\begin{align*}
K(p,a,q)^* &= \left( T(p^{-1}r) \phi(\alpha_r^{-1}(a)) T(q^{-1}r)^*\right)^* \\
&= T(q^{-1}r) \phi(\alpha_r^{-1}(a))^* T(p^{-1}r)^* \\
&= T(q^{-1}r) \phi(\alpha_r^{-1}(a^*)) T(p^{-1}r)^* \\
&= K(q,a^*,p).
\end{align*}

To see that $K$ is Toeplitz, assume first $pP\cap qP=\emptyset$; then it is clear that $rpP\cap rqP=r\cdot (pP\cap qP)=\emptyset$, and $K(p,a,q)=0=K(rp,\alpha_r(a),rq)$. Assume next $pP\cap qP=sP$; then $rpP\cap rqP=rsP$, and 
\begin{align*}
K(rp, \alpha_r(a), rq) &= T((rp)^{-1} rs) \phi(\alpha_{rs}^{-1}(\alpha_r(a))) T((rq)^{-1} rs)^* \\
&= T(p^{-1}s) \phi(\alpha_s^{-1}(a)) T(q^{-1}s)^* \\
&= K(p,a,q).
\end{align*}
Finally, the linearity of $K$ follows from that of $\phi$ and $\alpha_r$. 
\end{proof}

\begin{lemma}\label{lm.Kpos1} Suppose $p_1,\cdots, p_n\in P$ have a least common multiple $r\in \bigcap_{i=1}^n p_i P$, and let $a_1,\cdots,a_n\in \cA_r$. If $\phi$ is completely positive, then the operator matrix $[K(p_i, a_i^* a_j, p_j)]$ is positive definite.
\end{lemma} 
\begin{proof} For each $i$, write $r=p_i q_i$ for some $q_i\in P$ and $a_i=\alpha_r(b_i)$ for $b_i\in \cA$. Let $b_i=\alpha_r^{-1}(a_i)$. By Lemma \ref{lm.Kdef2}, 
\begin{align*}
K(p_i, a_i^* a_j, p_j) &= T(p_i^{-1} r) \phi(\alpha_r^{-1}(a_i^* a_j)) T(p_j^{-1} r)^* \\
&= T(p_i^{-1} r) \phi(b_i^* b_j) T(p_j^{-1} r)^*.
\end{align*}

Let $A=[\phi(b_i^* b_j)]$, $B=[b_i^* b_j]$, and $D$ be the $n\times n$ diagonal matrix with diagonal entries $T(p_i^{-1}r)$. Then 
\begin{align*}
[K(p_i, a_i^* a_j, p_j)] &= D A D^* 
= D \phi^{(n)}(B) D^*.
\end{align*}
Since $B$ is obviously positive and $\phi$ is completely positive, the matrix $[K(p_i, a_i^* a_j, p_j)]$ is positive definite. \end{proof}

\begin{lemma}\label{lm.EWF} 
Let $F=\{p_1,\cdots,p_n\}\subset P$ and for each subset $ W\subseteq F$ define \[E_{W,F}= \Big(\prod_{p\in W} E_p \Big)  \Big( \prod_{p\in F\backslash W} (I-E_p)\Big).\]
Then
\begin{enumerate}

\smallskip\item $E_{W,F}$ is an orthogonal projection.
\smallskip\item For each $p\in F\backslash W$ and $a\in\cA_p$, $aE_{W,F}=E_{W,F}a=0$. 
\smallskip\item For any $W_1\neq  W_2\subseteq F$, $E_{W_1, F} E_{W_2, F}=0$. 
\smallskip\item   \ $\sum_{W\subseteq F} E_{W,F}=I$.
\end{enumerate}
\end{lemma} 

\begin{proof} Since  $\{E_p: p\in F\}$ is a family of  commuting projections, their products (and products of their orthogonal complements $I-E_q$) are  projections. For $a\in \cA_p$ we have $a(I-E_p)=(I-E_p)a=0$ and thus $a E_{W,F}=E_{W,F}a=0$ whenever $p\notin W$. In particular, for any $W_1\neq  W_2\subseteq F$, there exists $p\in F$ that is an element for exactly one of $W_1, W_2$. Without loss of generality, assume $p\in W_1$ but $p\notin W_2$. We have $E_{W_1, F} E_p=E_{W_1, F}$, but $E_p E_{W_2, F}=0$. Therefore,  $E_{W_1, F} E_{W_2, F}=0$. Finally, $I=\prod_{p\in F} (E_p + (I-E_p))$. Expanding the product derives the desired equality. 
\end{proof} 

\begin{proposition}\label{prop.positive.ucp} $K$ is positive  if and only if $\phi$ is completely positive.
\end{proposition}

\begin{proof} Notice that  $K(e,a,e)=\phi(a)$ for every  $a\in\cA$. Assume that $K$ is positive. Given any positive $n\times n$ operator matrix $A=[a_{i,j}]\in M_n(\cA)$, we can find $B=[b_{i,j}]\in M_n(\cA)$ with $A=BB^*$. We have
\[\phi(A)=\phi(BB^*)=  \sum_{k=1}^n [\phi(b_{i,k}b_{j,k}^*)]_{n\times n} = \sum_{k=1}^n [K(e, b_{i,k} b_{j,k}^*,e)]_{n\times n}\geq 0,\]
proving that $\phi$ is completely positive.

Conversely, assume that $\phi$ is completely positive. Let $F=\{p_1,\cdots, p_n\} \subset P$ and $a_i\in \cA_{p_i}$. We need to prove that the operator matrix $[K(p_i, a_i^* a_j, p_j)]$ 
is positive definite.
 By Lemma \ref{lm.EWF}(4)  and the linearity of $K$, 
\[[K(p_i, a_i^* a_j, p_j)]=\sum_{W\subseteq F} [K(p_i, a_i^* E_{W,F} a_j, p_j)].\]
Hence it suffices to prove that for all $W\subseteq F$,
\[[K(p_i, a_i^* E_{W,F} a_j, p_j)]\geq 0.\]

For each $W\subseteq F$ and each $j\notin W$ we have  $E_{W,F}a_j=0$. Therefore, the entries of the matrix $[K(p_i, a_i^* E_{W,F} a_j, p_j)]$ vanish outside the $|W|\times |W|$ submatrix 
\[S_W=[K(p_i, a_i^* E_{W,F} a_j, p_j)]_{i,j\in W}.\]

If  $\bigcap_{p_i\in W} p_iP=\emptyset$, then $\prod_{p_i\in W} E_{p_i}=0$ so that $E_{W,F}=0$, in which case $S_W$ is the zero matrix. Otherwise, by the right LCM condition, $\bigcap_{p_i\in W} p_iP=rP$ for some $r\in P$ and  $c_i:=a_i E_{W,F}\in \cA_r$. Since $E_{W,F}$ is a projection, $c_i^* c_j = a_i^* E_{W,F} a_j$. Therefore, $S_W=[K(p_i, c_i^* c_j, p_j)]\geq 0$ by  Lemma \ref{lm.Kpos1}. 
\end{proof}

%\subsection{Conclusion of the proof of main theorem.}

We are now ready to prove   Theorem \ref{thm.dym.dilation}.

\begin{proof}[Proof of Theorem \ref{thm.dym.dilation}] Assume first  $\phi$ has a dilation; then it is a compression of a $\ast$-homomorphism, and thus must be unital completely positive.

Conversely, assume that $\phi$ is unital completely positive. Define a kernel system $K$ as in  Definition \ref{df.K}. Proposition \ref{prop.positive.ucp} proves that  $K$ is  positive. By  Proposition \ref{prop.Kbasic} and Theorem \ref{thm.kernel.dilation}, there exists a $\ast$-homomorphism $\pi:\cA\to\bh{K}$ and an isometric representation $V:P\to\bh{K}$ on some Hilbert space $\cK\supset \cH$, such that 
\[\phi(a)=P_\cH \pi(a)\bigg|_\cH \quad \text{ and } \quad T(p) = P_\cH V(p)\bigg|_\cH \]
 for all $a\in \cA$ and $p\in P$.
Moreover,  by Theorem~\ref{thm.kernel.dilation}, $(\pi,V)$ satisfies the weaker covariance relation $V(p)\pi(a)=\pi(\alpha_p(a))V(p)$. To see $(\pi,V)$ is in fact an isometric covariant representation, by Proposition \ref{prop.cov.equiv}, it suffices to prove that $V(p)V(p)^*=\pi(\alpha_p(1))$ for all $p\in P$. 

%Following the construction in Theorem \ref{thm.kernel.dilation},  

We claim that for every $\delta_{q,b}\otimes h\in \cK$,
\begin{equation}\label{eq:claim}
 V(p)^* \delta_{q,b}\otimes h=\begin{cases} 
\delta_{p^{-1} r, \alpha_{p^{-1}}(b)}\otimes T(q^{-1}r)^* h, &\mbox{ if } pP\cap qP=rP;\\0, & \mbox{ if } pP\cap qP=\emptyset .
\end{cases}
\end{equation}
Pick any $\delta_{s,c}\otimes k\in\cK$; then
\begin{align*}
\langle V(p)^* \delta_{q,b}\otimes h, \delta_{s,c}\otimes k\rangle 
&= \langle \delta_{q,b}\otimes h, V(p) \delta_{s,c}\otimes k\rangle \\
&= \langle \delta_{q,b}\otimes h, \delta_{ps,\alpha_p(c)}\otimes k\rangle \\
&= \langle K(ps, \alpha_p(c)^* b, q) h, k\rangle.
\end{align*}
Assume first $pP\cap qP=rP$. Then
\begin{align*}
\langle \delta_{p^{-1} r, \alpha_{p^{-1}}}(b)\otimes T(q^{-1}r)^* h, \delta_{s,c} \otimes k\rangle 
&= \langle K(s, c^* \alpha_{p^{-1}}(b), p^{-1}r) T(q^{-1}r)^* h, k\rangle \\
&= \langle K(ps, \alpha_p(c)^* b, r) T(q^{-1}r)^* h, k\rangle,
\end{align*}
where we have used the Toeplitz property of $K$ in the last line. By Lemma \ref{lm.Kpullout} we get
\begin{align*}
\langle \delta_{p^{-1} r, \alpha_{p^{-1}}}(b)\otimes T(q^{-1}r)^* h, \delta_{s,c} \otimes k\rangle 
&=
\langle K(ps, \alpha_p(c)^* b, r)T(q^{-1}r)^* h, k\rangle  \\
&= \langle K(ps, \alpha_p(c)^* b, q)h, k\rangle \\
&= \langle V(p)^* \delta_{q,b}\otimes h, \delta_{s,c}\otimes k\rangle.
\end{align*}
Therefore, $V(p)^* \delta_{q,b}\otimes h=\delta_{p^{-1} r, \alpha_{p^{-1}}(b)}\otimes T(q^{-1}r)^* h$, proving the first case in \eqref{eq:claim}. 
Assume next $pP\cap qP=\emptyset$. Then $psP\cap qP=\emptyset$, hence $\cA_{ps}\cA_q=\{0\}$, which implies that $\alpha_p(c)^* b =0$. Thus  $\langle V(p)^* \delta_{q,b}\otimes h, \delta_{s,c}\otimes k\rangle=0$ for arbitrary $\delta_{s,c}\otimes k\in \cK$. Hence $V(p)^* \delta_{q,b}\otimes h=0$,
 proving the second case in \eqref{eq:claim} and completing the proof of the claim.

In order to conclude that $V(p)V(p)^*=\pi(\alpha_p(1))$, since both sides are bounded linear operators, and since the dilation is minimal, it suffices to show that  $V(p)V(p)^* \delta_{q,b}\otimes h$ equals $\pi(\alpha_p(1)) \delta_{q,b}\otimes h = \delta_{q,\alpha_p(1) b}\otimes h$ for all $(q,b)\in \cX$ and $h\in \cH$.

Assume first $pP\cap qP=rP$. Then equation \eqref{eq:claim} shows that
\begin{align*}
V(p)V(p)^* \delta_{q,b}\otimes h &= V(p) \delta_{p^{-1} r, \alpha_{p^{-1}}(b)}\otimes T(q^{-1}r)^* h \\
&=  \delta_{r, \alpha_p(\alpha_{p^{-1}}( b))}\otimes T(q^{-1}r)^* h \\
&=\delta_{r, \alpha_p(1)b}\otimes T(q^{-1}r)^* h.
\end{align*}
Now let  $\delta_{s,c}\otimes k\in\cK$ and compute,
\begin{align*}
\langle \delta_{r, \alpha_p(1) b}\otimes T(q^{-1}r)^* h, \delta_{s,c}\otimes k\rangle &= \langle K(s, c^* \alpha_p(1) b, r) T(q^{-1}r)^* h, k\rangle \\
&= \langle K(s, c^* \alpha_p(1) b, q) h, k\rangle \\
&= \langle \delta_{q, \alpha_p(1) b}\otimes h, \delta_{s,c}\otimes k\rangle.
\end{align*}
where we have applied Lemma \ref{lm.Kpullout} in the second equality. So in this case we have 
\[
V(p)V(p)^* \delta_{q,b}\otimes h= \delta_{r, \alpha_p(1) b}\otimes T(q^{-1}r)^* h = \delta_{q, \alpha_p(1) b}\otimes h = : \pi(\alpha_p(1)) (\delta_{q,b}\otimes h).
\]

Assume now that $pP\cap qP=\emptyset$.  Equation \eqref{eq:claim} shows that $ V(p)^* \delta_{q,b}\otimes h=0$ and hence $V(p)V(p)^* \delta_{q,b}\otimes h =0$. On the other hand, $(q,b) \in \cX$ implies $\alpha_q(1) b =b$ so 
$\alpha_p(1)b = \alpha_p(1) \alpha_q(1) b=0$. Recall now that, by definition,
\[\langle \delta_{q,0}\otimes h, \delta_{r,a}\otimes k\rangle =\langle K(r,0,q) h, k\rangle,\]
for every $\delta_{r,a}\otimes k\in \cK_0$, and 
since $K(r,0,q)=0$,  we conclude that  in the present case
 $\pi(\alpha_p(1)) (\delta_{q,b}\otimes h) = 0 = \delta_{q,\alpha_p(1) b}\otimes h$.
Hence $V(p)V(p)^*=\pi(\alpha_p(1))$. 

 We know from Theorem \ref{thm.kernel.dilation} that the dilation satisfies the weak covariance condition and now that we have verified $V(p)V(p)^*=\pi(\alpha_p(1))$ we may use Proposition \ref{prop.cov.equiv} to conclude that $(\pi, V)$ is an isometric covariant representation. The minimality of $(\pi,V)$ and its uniqueness are established by Theorem \ref{thm.kernel.dilation}, completing the proof.
\end{proof}

\section{Dilations on Right LCM Semigroups}\label{sect.exLCM}

\subsection{Nica-covariance and dilation}

The main motivation of this paper comes from the recent development of dilation theory on right LCM semigroups in \cite{BLi2019} where it is shown that a contractive representation $T:P\to\bh{H}$ admits an isometric Nica-covariant dilation if and only if 
\begin{equation}\label{eq.dilation}
\sum_{U\subseteq F} (-1)^{|U|} T(\vee U)T(\vee U)^* \geq 0 \qquad \text{for all finite }F\subset P. \tag{$\diamond$}
\end{equation}

In trying to interpret this dilation result, we noticed the key role played by the semigroup dynamical system $(\cD_P, P, \alpha)$ based on the diagonal C*-algebra $\cD_P:=\overline{\lspan}\{E_p=V(p)V(p)^*: p\in P\}$ of the semigroup $\mathrm{C}^*$-algebra. Recall that $\cD_P:=\overline{\lspan}\{E_p=V(p)V(p)^*: p\in P\}$ is commutative by the Nica-covariance condition and that $P$ acts on $\cD_P$ by $\alpha_p(x)=V(p)xV(p)^*$. The system  $(\cD_P, P, \alpha)$ is the template we used for the definition of right LCM dynamical systems. Indeed,  Nica-covariance implies that 
\begin{align*}
\alpha_p(1)\alpha_q(1) &= E_p E_q 
= \begin{cases}
E_r = \alpha_r(1), & \mbox{ if } pP\cap qP=rP; \\
0, & \mbox{ if } pP\cap qP=\emptyset.
\end{cases}
\end{align*}
Using this one can easily check that $\alpha_p(\cD_P)$ is an ideal in $\cD_P$, showing  that $(\cD_P,P,\alpha)$ is a right LCM dynamical system in the sense of Definition~\ref{df.respectLCM}.

\begin{proposition}\label{prop.positive.DP} Let $\phi:\cD_P\to\bh{H}$ be a unital $\ast$-preserving linear map. Then the following are equivalent:

\begin{enumerate}
\item $\phi$ is positive\tcb{.}
\item $\phi$ is  completely positive.
\item For each finite $F\subset P$ and $W\subseteq F$,
 \[\phi\left(\bigg( \prod_{f\in W} E_f \bigg)\bigg( \prod_{f\in F\backslash W} (I-E_f) \bigg)\right)\geq 0.\]
\item For each finite $F\subset P$
\[\phi\bigg(\prod_{f\in F} (I-E_f) \bigg)\geq 0.\]
\end{enumerate}
\end{proposition} 

\begin{proof} Recall that  $\cD_P$ is a commutative $\mathrm{C}^*$-algebra, so (1) and (2) are equivalent. 

For a finite subset $F\subset P$ and $W\subseteq F$, define 
\[
E_{W,F}=\bigg(\prod_{p\in W} E_p \bigg)\bigg( \prod_{p\in F\backslash W} (I-E_p)\bigg).
\]
For fixed $F$, one can apply Lemma \ref{lm.EWF} to show that the family $\{E_{W,F}\}_{W\subseteq F}$ consists of orthogonal projections with mutually orthogonal ranges that sum up to $I$. An element $x\in \lspan\{E_p: p\in F\}$ can therefore be written as $x=\sum_{W\subseteq F} \lambda_{W,F} E_{W,F}$, and it is positive if and only if $\lambda_{W,F}\geq 0$ for all $W\subseteq F$ such that $E_{W,F} \neq 0$. 
Therefore, $\phi$ is positive if and only if $\phi(E_{W,F})\geq 0$ for all finite $F\subset P$ and $W\subseteq F$, which proves the equivalence between (1) and (3). 

It is easy to see that (3)  implies (4) by taking $W=\emptyset$. 
To see the converse notice that, by  Nica-covariance, either $\prod_{f\in W} E_f =0$, in which case $\phi(E_{W,F})=0$ and (3) holds, or else $\prod_{f\in W} E_f =E_r$ for  $rP = \bigcap_{f\in W} fP$. In the latter case, for each $f\in F\backslash W$, 
\[E_r E_f =\begin{cases} E_s, &\mbox{ if } rP\cap fP=sP; \\
0, &\mbox{ if } rP\cap fP=\emptyset .\\
\end{cases} \]
Therefore, 
\[V(r)^*(I-E_f) = \begin{cases} (I-E_{r^{-1}s(f)}) V(r)^*, &\mbox{ if } rP\cap fP=s(f) P; \\
V(r)^*, &\mbox{ if } rP\cap fP=\emptyset .\\
\end{cases} \]
Next let $F_0=\{r^{-1}s(f): s(f) P=rP\cap fP, \  f\in F\backslash W, \text{ and } rP\cap fP \neq \emptyset\}$. We have
\begin{align*}
E_r \prod_{f\in F\backslash W} (I-E_f) &= V(r) \prod_{f\in F_0} (I-E_f)V(r)^*. \\
&= \alpha_r( \prod_{f\in F_0} (I-E_f)).
\end{align*}
By (4), $\phi( \prod_{f\in F_0} (I-E_f))\geq 0$. Therefore, 
\begin{align*}
\phi(E_{W,F}) &= \phi(\alpha_r(\prod_{f\in F_0} (I-E_f))) 
= T(r) \phi(\prod_{f\in F_0} (I-E_f)) T(r)^* \geq 0,
\end{align*}
so (3) holds in this case too.
\end{proof} 

{For each contractive representation  $T:P\to\bh{H}$ one can define $\phi_T(E_p)=T(p)T(p)^*$. Since the projections $E_p$ are linearly independent this can be extended by linearity to all of $\lspan\{E_p=V(p)V(p)^*: p\in P\}$. It is not immediate that $\phi_T$ can be extended to a contractive positive map on $\cD_P$. The following lemma gives a criterion for when this is possible.}

\begin{lemma} \label{lem:extension} Suppose $T:P\to\bh{H}$ is a contractive representation, let $\phi_T(E_p):=T(p)T(p)^*$ and extend $\phi_T$ linearly to $\lspan\{E_p=V(p)V(p)^*: p\in P\}$. The following are equivalent
\begin{enumerate}
\item The map $\phi_T$  extends to a positive map on $\cD_P=\overline{\lspan}\{E_p=V(p)V(p)^*: p\in P\}$.
\item For each finite $F\subset P$ and $W\subseteq F$,
 \[\phi_T\left(\bigg( \prod_{f\in W} E_f \bigg)\bigg( \prod_{f\in F\backslash W} (I-E_f) \bigg)\right)\geq 0.\]
\item For each finite $F\subset P$
\[\phi_T\bigg(\prod_{f\in F} (I-E_f) \bigg)\geq 0.\]
\end{enumerate}
\end{lemma} 

\begin{proof} The equivalence among (2) and (3) follows in the same way as the proof of Proposition \ref{prop.positive.DP}. It is obvious that (1) implies (2) since each $\bigg( \prod_{f\in W} E_f \bigg)\bigg( \prod_{f\in F\backslash W} (I-E_f) \bigg)\geq 0$. For the converse, pick any $x=\sum_{p\in F} \lambda_p E_p$ with $\|x\| =1$. By Lemma~\ref{lm.EWF}, 
\[
 x = \left(\sum_{W\subseteq F} E_{W,F}\right) x = \sum_{W\subseteq F} \sum_{p\in F} \lambda_p E_{W,F} E_p= \sum_{W\subseteq F} (\sum_{p\in W} \lambda_p) E_{W,F}.
\]
Let $\mu_W=\sum_{p\in W} \lambda_p$. Since $\{E_{W,F}\}$ is a collection of pairwise orthogonal projections, 
\[
\|x\|=\max_{W\subset F, E_{W,F}\neq 0} |\mu_W|  =  1.
\]
To prove $\phi_T$ can be extended by continuity to $\cD_P$, it suffices to prove that $\|\phi_T(x)\|\leq 1$ and thus $\phi_T$ is contractive. This is equivalent to proving that the $2\times 2$ operator matrix $\begin{bmatrix} I & \phi_T(x) \\ \phi_T(x)^* & I\end{bmatrix}$ is positive definite. Since
\[ 
\begin{bmatrix} I & \phi_T(x) \\ \phi_T(x)^* & I\end{bmatrix}
= \sum_{W\subseteq F} \begin{bmatrix} \phi_T(E_{W,F}) & \mu_W \phi_T(E_{W,F}) \\ \overline{\mu_W} \phi_T(E_{W,F}) & \phi_T(E_{W,F})\end{bmatrix}
\]
and  each $\phi_T(E_{W,F})$ is positive by (2), this operator matrix is always positive.
\end{proof}

%Let $(\phi,T)$ be a contractive covariant  representation of a right LCM dynamical system $(\cD_P, \alpha, P)$. We have $\phi(E_r)=T(r)T(r)^*$. 
As an application of our Theorem~\ref{thm.dym.dilation}, we recover
 the following dilation theorem from \cite{BLi2019}.

\begin{theorem}\cite[Theorem~3.9]{BLi2019}\label{thm.LCM.Dilation} Let $T:P\to\bh{H}$ be a contractive unital representation of a right LCM semigroup. The following are equivalent:
\begin{enumerate}
\item $T$ has an isometric Nica-covariant dilation;
\item For any finite set $F\subset P$, 
\[\sum_{U\subseteq F} (-1)^{|U|} T(\vee U)T(\vee U)^* \geq 0.\]
\end{enumerate}
\end{theorem}
\begin{proof}
Suppose (1) holds. We may assume that the dilation $V$ is minimal and hence $\cH$ is co-invariant. Let $E_f := V_fV_f^*$. From the inclusion-exclusion principle, for every  finite $F\subset P$,
\[\prod_{f\in F} (I-E_f) = \sum_{U\subseteq F} (-1)^{|U|} E_{\vee U},\]
where we set $E_{\vee U}=0$ when $\bigcap_{p\in U} pP=\emptyset$, and $E_{\vee U}=E_r$ when $\bigcap_{p\in U} pP=rP$.  That (1) implies (2) is now easy to see because $\sum_{U\subseteq F} (-1)^{|U|} T(\vee U)T(\vee U)^*$  is the compression of $\prod_{f\in F} (I-E_f) $ to the co-invariant subspace $\cH$. 

Suppose now that $T$ is a contractive unital representation of $P$ such that (2) holds and define $\phi(E_r):=T(r)T(r)^*$.  Then condition (3) in Lemma~\ref{lem:extension} also holds. By condition (1) in Lemma~\ref{lem:extension} we can  extend $\phi$ to a contractive representation of $\cD_P$, which we also denote by $\phi$, such that $(\phi,T)$ is a contractive covariant pair. Since $T$ satisfies (2), Proposition \ref{prop.positive.DP} implies that $\phi$ is  positive, hence completely positive.  Theorem~\ref{thm.dym.dilation} does the rest.
\end{proof}

\subsection{Nica spectrum and boundary quotient}

The diagonal $\cD_P$ is a unital commutative $\mathrm{C}^*$-algebra, which is isomorphic to $C(\Omega_P)$ for a compact Hausdorff space $\Omega_P$. The space $\Omega_P$ is called the Nica-spectrum of $P$, and one can refer to \cite{Nica1992, Laca1999, CrispLaca2007} for more detailed descriptions. The action $\alpha$ by injective endomorphisms of  $\cD_P$ induces an action $\hat\alpha$ of $P$ by surjective maps of $\Omega_P$ to itself. A compact subset $K\subset \Omega_P$ is called {\em \tcb{$\hat\alpha$}-invariant} if both $K$ and its complement are invariant, i.e. $\hat\alpha_p(K) \subseteq K$ and $\hat\alpha_p(\Omega_P\backslash K) \subseteq \Omega_P\backslash K$. Notice that in the case of a group one of the inclusions implies the other, but for semigroup actions, both inclusions have to be required.
%, where both inclusions can be proper because $\alpha$ is not surjective.

\begin{proposition} Let $K\subset \Omega_P$ be a closed $\hat \alpha$-invariant  subset of $\Omega_P$. Then $(C(K),P,\alpha)$ is also a right LCM dynamical system.
\end{proposition} 

\begin{proof} Let $\Omega_{P,p}=\hat{\alpha}_p(\Omega_P)$ and $K_p=\hat{\alpha}_p(K)$.  Since $K$ and its complement are invariant for $\alpha$, 
\[K_p\cap K_q \subset \Omega_{P,p}\cap \Omega_{P,q}\cap K \quad p,q\in P.\]
If $pP\cap qP=\emptyset$, then $\Omega_{P,p}\cap \Omega_{P,q}=\emptyset$ and also $K_p\cap K_q$. If $pP\cap qP=rP$, then
\[\Omega_{P,p}\cap \Omega_{P,q}\cap K=\Omega_{P,r}\cap K=K_r\subseteq K_p\cap K_q.\]
and thus $K_p\cap K_q=K_r$. Therefore, $\cA_p=\alpha_p(C(K))=C(K_p)$ and $\cA_p\cA_q\subseteq C(K_p\cap K_q)$, which becomes $\{0\}$ when $pP\cap qP=\emptyset$ and becomes $C(K_r)$ when $pP\cap qP=rP$. 
\end{proof} 

Observe that $C(\Omega_P)=\cD_P=\overline{\lspan}\{E_p=V(p)V(p)^*\}$. Thus, for every contractive representation $(\phi,T)$ of the dynamical system, \[\phi(E_p)=\phi(\alpha_p(1))=T(p)T(p)^*.\]
This relation remains true when we consider $(C(K),P,\alpha)$ for any  invariant subset $K$. 
\begin{theorem} Let $K$ be a closed invariant  subset of $\Omega_P$ and $(\phi,T)$ be a contractive covariant representation of $(C(K),P,\alpha)$. Then the following are equivalent:
\begin{enumerate}
\item There exists an isometric covariant dilation $(\pi, V)$ of $(\phi,T)$.
\item $\phi$ is positive.
\item For each finite $F\subseteq P$, 
\[\sum_{U\subseteq F} (-1)^{|U|} T(\vee U)T(\vee U)^* \geq 0.\]

\end{enumerate} 
\end{theorem} 
\begin{proof}
That (1) implies (2) is easy to see, and (2) $\iff$ (3) follows from covariance and   Proposition~\ref{prop.positive.DP}. 
Suppose now (2) holds. Since $C(K)$ is commutative $\phi$ is completely positive, hence 
(1) holds by Theorem~\ref{thm.dym.dilation}.
\end{proof}

An important invariant subspace of $\Omega_P$ is the boundary spectrum $\partial \Omega_P$, which consists of the closure of the set of maximal elements in $\Omega_P$ \cite[Definition 5.1]{BRSW2014}, see also \cite{Laca1999}. The boundary quotient corresponding to the boundary spectrum can be characterized via foundation sets. By definition, a finite set $F\subset P$ is a \emph{foundation set} if for all $p\in P$, there exists $f\in F$ such that $fP\cap pP\neq \emptyset$. Let $I_\infty$ be the ideal of $\cD_P$ generated by the projections $ \prod_{f\in F}(I-E_f)$ for every foundation set $F$. Then  
\[
C(\partial \Omega_P) =  \cD_P / I_\infty.
\]
The \emph{boundary quotient} $\mathrm{C}^*$-algebra $\cQ(P)$ is the universal $\mathrm{C}^*$-algebra generated by Nica-covariant representations $V$ with the additional requirement that 
\begin{equation}\label{eq:boundaryrelations}
\prod_{f\in F}(I-V_fV_f^*)=0 \mbox{ for every foundation set }F.
\end{equation}
This can also be realized as the cross-product $\partial \cD_P \rtimes P$. For more detailed studies of the boundary quotient $\mathrm{C}^*$-algebras, one may refer to \cite{Laca1999, CrispLaca2007, Starling2015}. 

\begin{example} Consider the free semigroup $\bF_n^+$. A Nica-covariant representation $V$ is  uniquely determined by its value on the $n$ generators $V_1, \cdots, V_n$. In this case, the Nica-covariance condition is equivalent to saying that the generating isometries $V_1,\cdots, V_n$ have orthogonal ranges, or in short, 
\[\sum_{i=1}^n V_i V_i^*\leq I.\] 
The universal $\mathrm{C}^*$-algebra (or the semigroup $\mathrm{C}^*$-algebra for $\bF_n^+$) is thus the Toeplitz-Cuntz algebra $\mathcal{TO}_n$. 
However, the boundary quotient requires further that $\sum_{i=1}^n V_i V_i^*=I$, and the boundary quotient semigroup $\mathrm{C}^*$-algebra for $\bF_n^+$ is the Cuntz algebra $\cO_n$. 
\end{example} 

A contractive representation of the boundary quotient dynamical system $(C(\partial \Omega_P), P, \alpha)$ is uniquely determined by a contractive representation $T:P\to\bh{H}$ with the additional requirement that for every foundation set $F\subset P$, 
\[\sum_{U\subseteq F} (-1)^{|U|} T(\vee U)T(\vee U)^* = 0.\]

As a corollary of Theorem~\ref{thm.LCM.Dilation}, we establish a condition for  a contractive representation to be dilated to an isometric representation of the boundary quotient. 

\begin{corollary}\label{cor.QDilation} Let $T:P\to\bh{H}$ be a contractive representation of a right LCM semigroup $P$ such that for each foundation set $F\subset P$, 
\[\sum_{U\subseteq F} (-1)^{|U|} T(\vee U)T(\vee U)^* = 0.\]
Then, the following are equivalent:
\begin{enumerate}
\item There exists a Nica covariant isometric dilation $V:P\to\bh{K}$ satisfying \eqref {eq:boundaryrelations}\tcb{.}
\item For each finite $F\subseteq P$, 
\[\sum_{U\subseteq F} (-1)^{|U|} T(\vee U)T(\vee U)^* \geq 0.\]
\end{enumerate}
\end{corollary} 

\section{Dilation on $C(\Omega_P)\otimes \cA$}\label{sect.example}

In this section we build more examples of right LCM dynamical systems, starting from automorphic actions. Throughout  we fix a right LCM semigroup $P$ with trivial unit group $P^\ast=\{e\}$, so that whenever $pP\cap qP=rP$, the choice of $r$ is unique. As before, $\Omega_P$ denotes its Nica-spectrum and $\partial \Omega_P$ denotes the boundary of  $\Omega_P$.

\subsection{A dynamical system on $C(\Omega_P)\otimes \cA$}
Suppose  $\beta$ is a $\ast$-automorphic $P$-action on the unital C*-algebra $\cA$. The system $(\cA, P, \beta)$ is not necessarily a right LCM dynamical system, in fact it cannot be when  there exist elements $p,q$ such that $pP\cap qP=\emptyset$, because then $\beta_p(\cA)\beta_q(\cA)=\cA\neq \{0\}$. 
Our first goal is to construct a right LCM dynamical system on the $\mathrm{C}^*$-algebra  $\widetilde{\cA}:=C(\Omega_P)\otimes \cA $ of continuous $\cA$-valued functions on $\Omega_P$, and then
give conditions under which unital completely positive maps on $\cA$ can be lifted to $\widetilde{\cA}$. 
Recall the natural right LCM dynamical system $(C(\Omega), P, \alpha)$ where $\alpha_p(V_qV_q^*)=V_{pq} V_{pq}^*$, and define $\widetilde{\alpha}$ on $\widetilde{\cA}$ by 
\[\widetilde{\alpha_p}(f\otimes a) = \alpha_p(f)\otimes \beta_p(a).\]
In other words, $\widetilde{\alpha}=\alpha \otimes \beta$, which is clearly a $P$-action on $\widetilde{\cA}$ by $\ast$-endomorphisms. 

\begin{lemma} The semigroup dynamical system $(\widetilde{\cA}, P, \widetilde{\alpha})$ is a right LCM dynamical system. 
\end{lemma} 

\begin{proof} Since $\beta_p$ is a $\ast$-automorphism, the image $\widetilde{\alpha_p}(\widetilde{\cA})$ is $\alpha_p(C(\Omega_P))\otimes \cA$. Since $(C(\Omega_P),P,\alpha)$ is a right LCM dynamical system, by Proposition \ref{prop.equivLCM},

\[\alpha_p(C(\Omega_P)) \alpha_q(C(\Omega_P)) = \begin{cases}
\alpha_r(C(\Omega_P)), &\mbox{ if } pP\cap qP=rP; \\
\{0\}, &\mbox{ if } pP\cap qP=\emptyset. 
\end{cases}
\]
This relation is preserved by   tensor products, and another application of Proposition \ref{prop.equivLCM} shows that
$(\widetilde{\cA}, P, \widetilde{\alpha})$ is a right LCM dynamical system.
\end{proof}

Next, given a unital completely positive map $\phi:\cA\to\bh{H}$ and a contractive representation $T:P\to\bh{H}$, we hope to construct a covariant representation $(\widetilde{\phi}, T)$ for the right LCM dynamical system $(\widetilde{\cA}, P, \widetilde{\alpha})$. Note  that $\phi$ and $T$ do not need to satisfy any covariance relation by themselves. Recall $C(\Omega_P)\cong \cD_P$ is the closed linear span of  the projections $E_p=V_p V_p^*$ and  define 
\[\widetilde{\phi}(E_p\otimes a) = T(p) \phi(\beta_p^{-1}(a)) T(p)^*.\] 
For every finite $F\subset P$, we define a map $\phi_F:\cA\to\bh{H}$ as follows. If a subset $U \subset F$ has an upper bound, we let $s_U $ be the least upper bound of $U$, so that $s_U P=\bigcap_{p\in U} pP$ and  $s_U$ is uniquely determined because we are assuming $P$ has no nontrivial units. If $U$ has no upper bound, so that $\bigcap_{p\in U} pP=\emptyset$, then by convention we say $\tcb{s_U} =\infty$ and we let $T(\infty) =0$. Then we define
\[\phi_F(a) = \sum_{U\subseteq F} (-1)^{|U|} T(s_U) \phi(\beta_{s_U}^{-1} (a)) T(s_U)^*.\]

%Here, $s_U P=\bigcap_{p\in U} pP$, and by convention, when $\bigcap_{p\in U} pP=\emptyset$ and $T(s_U)=0$. Since we assume that $P$ has only trivial unit, there is no ambiguity in the choice of $s_U$, and thus $\phi_F$ are well defined.

\begin{proposition}\label{prop.tensor.UCP} Suppose $\beta$ is a $\ast$-automorphic $P$-action on $\cA$,  
 $\phi:\cA\to\bh{H}$ is a unital completely positive map and $T:P\to\bh{H}$ is a contractive representation. Let 
 $\widetilde{\phi}$ %(E_p\otimes a) = T(p) \phi(\beta_p^{-1}(a)) T(p)^*$ 
 and 
$\phi_F$ 
%for (a) = \sum_{U\subseteq F} (-1)^{|U|} T(s_U) \phi(\beta_{s_U}^{-1} (a)) T(s_U)^*\]
 be as above.
 The pair $(\widetilde{\phi},T)$ can be extended to a contractive covariant representation of $(\widetilde{\cA}, P, \widetilde{\alpha})$ with unital completely positive $\widetilde{\phi}$ if and only if  $\phi_F$ is completely positive for all finite $F\subset  P$. 
\end{proposition} 

\begin{proof} 
For each finite subset $F \subset P$ and $a\in \cA$ we have
\begin{align*}
\phi_F(a) &=   \sum_{U\subseteq F} (-1)^{|U|} T(s_U) \phi(\beta_{s_U}^{-1} (a)) T(s_U)^* \\
&=  \sum_{U\subseteq F} (-1)^{|U|} \widetilde{\phi}(E_{s_U} \otimes a) \\
&=  \widetilde{\phi}(E_F \otimes a).
\end{align*}
where $E_F = \sum_{U\subseteq F} (-1)^{|U|} E_{s_U}=\prod_{p\in F} (I - E_p)$ is an orthogonal projection in $C(\Omega_P)$. If
$(\widetilde{\phi},T)$ is  a contractive covariant representation with $\widetilde{\phi}$ a unital completely positive map,  then $\phi_F$ must be completely positive. 

For the converse, we first show that for every $p,q\in P$ and $a\in \cA$, 
\begin{align*}
\widetilde{\phi}(\widetilde{\alpha_p}(E_q\otimes a)) &= \widetilde{\phi}(E_{pq}\otimes \beta_p(a)) \\
&= T(p)T(q) \phi(\beta_{pq}^{-1}(\beta_p(a)) T(q)^* T(p)^* \\
&= T(p) T(q) \phi(\beta_q^{-1}(a)) T(q)^* T(p)^* \\
&= T(p) \widetilde{\phi}(E_q\otimes a) T(p)^*.
\end{align*}
Therefore, for any $x\in\lspan\{E_q\otimes a: q\in P, a\in \cA\}$, 
\[\widetilde{\phi}(\widetilde{\alpha_p}(x)) = T(p) \widetilde{\phi}(x) T(p)^*.\]
Next, for each $W\subseteq F$ we  define 
\[\phi_{W,F}(a) = \widetilde{\phi}\Big(\Big(\prod_{e\in W} E_e  \prod_{f\in F\backslash W} (I-E_f) \Big)\otimes a\Big).\]
We claim that $\phi_{W,F}$ is completely positive.
If $\bigcap_{e\in W} eP=\emptyset$, we have $\prod_{e\in W} E_e=0$ so the claim is trivially satisfied. Otherwise, $\bigcap_{e\in W} eP=rP$ for some $r\in P$, and $\prod_{e\in W} E_e=E_r$. Define $F_0=\{r^{-1}s: sP=rP\cap fP, f\in F\backslash W\}$. Then, for each $f\in F\backslash W$, either $fP\cap rP=\emptyset$ and $E_r (I-E_f)=E_r$, or $fP\cap rP=sP$ so that $E_r(I-E_f) = E_r - E_s = \alpha_r (I-E_{r^{-1}s})$. Therefore, 
\begin{align*}
\phi_{W,F}(a) &= \widetilde{\phi}(\alpha_r(\prod_{p\in F_0} (I-E_p))\otimes a) \\
&= \widetilde{\phi}(\widetilde{\alpha_r}( E_{F_0} \otimes \beta_r^{-1}(a) )) \\
&= T(r) \phi_{F_0}(\beta_r^{-1}(a)) T(r)^*.
\end{align*}  
Since $\phi_{F_0}$ is completely positive, composing with a $\ast$-automorphism $\beta_r^{-1}$ and conjugating with a contraction $T(r)$ yields another completely positive map $\phi_{W,F}$.  This completes the proof of the claim.

Let $C_0(\Omega_P)=\lspan\{E_p\}$, which is dense in $C(\Omega_P)$. We first extend  $\widetilde{\phi}$ to $\widetilde{\cA_0}=C_0(\Omega_P)\otimes \cA$. This extension is clearly unital because
$\widetilde{\phi}(I\otimes 1)=T(e) \phi(1) T(e)^* = I$. To prove  that $\widetilde{\phi}$ is completely positive, pick any $x_1,\cdots, x_n\in \widetilde{\cA_0}$. Since $C_0(\Omega_P)=\lspan\{E_p\}$, one can find a finite subset $F\subset P$ and elements $\{a_{f,i}\}_{f\in F, 1\leq i\leq n} \subset \cA$, such that \[x_i = \sum_{f\in F} E_f \otimes a_{f,i}.\]
For each $W\subseteq F$, let $E_{W,F}=\prod_{e\in W} E_e \prod_{f\in F\backslash W} (I-E_f)$. From Lemma~\ref{lm.EWF}, $\{E_{W,F}\}_{W\subseteq F}$ are orthogonal projections and they are pairwise orthogonal. Moreover, for each $f\in F$, the projection $E_f$ decomposes as 
\[E_f = \sum_{f\in W} E_{W,F}.\]
Therefore, rearranging and combining terms, we can rewrite \[x_i=\sum_{W\subseteq F} E_{W,F}\otimes a_{W,F,i}\]
for some $a_{W,F,i}\in \cA$, and thus
\[x_i^* x_j = \sum_{W\subseteq F} E_{W,F}\otimes a_{W,F,i}^* a_{W,F,j}.\]
It follows that the matrix
\begin{align*}
[\widetilde{\phi}(x_i^*x_j)] &= \sum_{W\subseteq F} [\widetilde{\phi} (E_{W,F}\otimes a_{W,F,i}^* a_{W,F,j})] \\
 &=  \sum_{W\subseteq F} [\phi_{W,F}(a_{W,F,i}^* a_{W,F,j})] 
\end{align*}
 is positive because $\phi_{W,F}$ is completely positive for each $W$. This shows that $\widetilde{\phi}$ is unital completely positive on $\widetilde{\cA_0}$. Since completely positive maps are also completely contractive, $\widetilde{\phi}$ can be extended by continuity to a unital completely positive map on $\widetilde{\cA}$. 
\end{proof} 

When we apply Theorem \ref{thm.dym.dilation} to $(\widetilde{\phi},T)$ we obtain the following strengthening of Theorem~\ref{thm.LCM.Dilation}.

\begin{theorem}\label{thm.tensor.dilation} Suppose $\beta$ is a $\ast$-automorphic $P$-action on  the  unital $\mathrm{C}^*$-algebra $\cA$. Let  $\phi:\cA\to\bh{H}$ be a unital completely positive map on  $\cA$ and let $T:P\to\bh{H}$ be a contractive representation of the right LCM semigroup $P$. The following are equivalent:
\begin{enumerate}
\item There exist a $\ast$-homomorphism $\pi:\cA\to\bh{K}$ and an isometric representation $V:P\to\bh{K}$ on a Hilbert space $\cK\supset \cH$ such that for $a,b\in \cA$ and  $p,q\in P$, if  $pP\cap qP=rP$, then
\begin{equation}\label{eqn:multipl-dilation}
V(p)\pi(\beta_p^{-1}(a))V(p)^* V(q)\pi(\beta_q^{-1}(b))V(q)^* =
V(r) \pi(\beta_r^{-1}(ab)) V(r)^*;
\end{equation}
and if $pP\cap qP=\emptyset$, then $V(p)$ and $V(q)$ have orthogonal ranges.

\noindent Moreover, $\pi$ and $V$ are dilations for $\phi$ and $T$, respectively, and $\cH$ is co-invariant for $V$. 
\item For each finite $F\subset P$, the map $\phi_F:\cA\to\bh{H}$ defined by \[
\phi_F(a) = \sum_{U\subseteq F} (-1)^{|U|} T(s_U) \phi(\beta_{s_U}^{-1} (a)) T(s_U)^*
\]
 is completely positive.
\end{enumerate}
\end{theorem}

\begin{proof} Assume first that (2) holds and let $(\widetilde{\phi},T)$ be the contractive covariant representation of the system $(\widetilde{\cA}, P, \widetilde{\alpha})$ from Proposition \ref{prop.tensor.UCP}. The map $\widetilde{\phi}$ is unital and completely positive, therefore, by Theorem \ref{thm.dym.dilation}, $(\widetilde{\phi},T)$ admits an isometric covariant dilation $(\widetilde{\pi},V)$ on $\bh{K}$ for $\cK\supset\cH$ in which $\cH$ is co-invariant for $V$. Define $\pi:\cA\to\bh{K}$ by $\pi(a)=\widetilde{\pi}(I\otimes a)$. One can easily verify that $\pi$ is a $\ast$-homomorphism that is a dilation for $\phi$ and that
{$\widetilde{\pi}(E_p\otimes a)=V(p)\pi(\beta_p^{-1}(a)) V(p)^*$.} Notice that
\[ E_p\otimes a \cdot E_q\otimes b = \begin{cases}
E_r\otimes ab, &\mbox{ if } pP\cap qP=rP; \\
0, &\mbox{ if } pP\cap qP=\emptyset. 
\end{cases}
\]
If $pP\cap qP=rP$, then
\begin{align*}
 V(p)\pi(\beta_p^{-1}(a) )V(p)^* V(q)\pi(\beta_q^{-1}(b))V(q)^* &=\widetilde{\pi}(E_p\otimes a) \widetilde{\pi}(E_q\otimes b)
 \\
&= 
\widetilde{\pi}(E_r\otimes ab) \\
&=
V(r) \pi(\beta_r^{-1}(ab)) V(r)^* ,
\end{align*}
and if $pP\cap qP=\emptyset$, then 
\[
\widetilde{\pi}(E_p\otimes 1) \widetilde{\pi}(E_q\otimes 1)=V(p)V(p)^* V(q)V(q)^*=0
\] so that $V(p)$ and $V(q)$ have orthogonal ranges. 

Conversely, assume now (1) holds. Then $\pi$ is a $\ast$-homomorphism and thus unital completely positive. Therefore, one can extend $\pi$ and $V$ to obtain a contractive covariant representation $(\widetilde{\pi},V)$ of $(\widetilde{\cA},P,\widetilde{\alpha})$. From the construction, $\widetilde{\pi}(E_p\otimes a)=V(p) \pi(\beta_p^{-1}(a)) V(p)^*$, and thus by equation \eqref{eqn:multipl-dilation}, $\widetilde{\pi}$ is multiplicative. Therefore, $\widetilde{\pi}$ is a $\ast$-homomorphism of $\widetilde{\cA}$. For each finite $F\subset P$, we have
\[\widetilde{\pi}_F(a)=\widetilde{\pi}((\prod_{f\in F} (I-E_f)) \otimes a)=
\sum_{U\subseteq F} (-1)^{|U|} V(s_U) \pi(a) V(s_U)^*.
\]
which shows that $\widetilde{\pi}_F$ is a (non-unital) $\ast$-homomorphism of $\cA$. Since $\pi$ and $ V$ are dilations of $\phi$ and $T$ and $\cH$ is co-invariant for $V$, projecting to the corner of $\cH$  obtains $\phi_F(a)$, which is thus completely positive. 
\end{proof}

\begin{remark} In the special case when $\cA=\mathbb{C}$, a unital completely positive map $\phi$ on $\cA$ is uniquely determined by $\phi(x)=xI$. In this case, condition (2) is  reduced to that in  Theorem \ref{thm.LCM.Dilation}. 
\end{remark} 

 One can use a similar technique to obtain a strengthened version of Corollary \ref{cor.QDilation}.

\begin{corollary}\label{cor.tensor.QP} Suppose that in addition to the conditions in Theorem~\ref{thm.tensor.dilation}, we also have that, 
\begin{equation}\label{cor.tensor.QP.eq1}
\sum_{U\subseteq F} (-1)^{|U|} T(s_U)\phi(\beta_{s_U}^{-1}(a))T(s_U)^* =0
\end{equation}
 for every (finite) foundation set $F\subset P$. Let $V$ be the
resulting dilation from Theorem \ref{thm.tensor.dilation}. Then 
\begin{equation}\label{cor.tensor.QP.eq2}
\sum_{U\subseteq F} (-1)^{|U|} V(s_U)\pi(\beta_{s_U}^{-1}(a))V(s_U)^* =0
\end{equation}
for every (finite) foundation set $F\subset P$. 
\end{corollary} 

\begin{proof} Notice first that the kernel of  the quotient $C(\Omega_P)\otimes \cA \to C(\partial \Omega_P)\otimes \cA$ is the ideal generated by $E\otimes a$ where the
projection $E$ is associated to a foundation set. Hence, condition \pref{cor.tensor.QP.eq1} ensures that the map $\tilde \phi$ from the pair  $(\tilde{\phi},T)$ obtained in Proposition \ref{prop.tensor.UCP} factors through this quotient.  This gives a pair $(\dot{\tilde{\phi}}, T)$ for the  LCM system on  the quotient $\dot{\tilde \cA}: =C(\partial \Omega_P)\otimes \cA$. 

The equivalent conditions in Theorem~\ref{thm.tensor.dilation} ensure that $\tilde{\phi}$ is completely positive and hence so is $\dot{\tilde{\phi}}$. Therefore, we can dilate $(\dot{\tilde{\phi}},T)$ to an isometric covariant representation $(\rho,V)$ of the $\dot{\tilde \cA} $ system. Letting $\pi(a) := \rho(I\otimes a)$ for $a \in \cA$ gives a pair $(\pi,V)$ satisfying the required equation
\eqref{cor.tensor.QP.eq2}. 
\end{proof} 

\subsection{Examples from $\mathbb{F}_k^+$} 

\begin{example}\label{ex.dilation.Fk} Let us consider the case when $P=\mathbb{F}_2^+$. Suppose $\cA$ is any unital $\mathrm{C}^*$-algebra with a unital completely positive map $\phi:\cA\to\bh{H}$ and two $\ast$-automorphisms $\beta_1,\beta_2$. Suppose that $T_1,T_2\in\bh{H}$ are two contractions such that for all $a\in A$, 
\[\phi(a) = T_1 \phi(\beta_1^{-1}(a)) T_1^* + T_2 \phi(\beta_2^{-1}(a)) T_2^* \]
We would like to claim that $\tcb{(\phi, T)}$ can be dilated to a $\ast$-homomorphism $\pi:\cA\to\bh{K}$ and isometries $V_1,V_2\in\bh{K}$ such that 
\[\pi(a) = V_1 \pi(\beta_1^{-1}(a)) V_1^* + V_2 \pi(\beta_2^{-1}(a)) V_2^*\]

We first build $C(\partial \Omega_P)\otimes \cA$ as an inductive limit of the following system: let $\cA_1=\cA$, and $\cA_{n+1} = \cA_n \oplus \cA_n$ with connecting map $\varphi_n:\cA_n \to \cA_{n+1}$ by $\varphi_n(a)=a\oplus a$. One can verify that the inductive limit $\widetilde{\cA}\cong C(X)\otimes \cA$, where $X$ is the Cantor set. 

Define two $\ast$-endomorphisms $\alpha_1, \alpha_2$ on $\widetilde{\cA}$ by $\alpha_1(a)=\beta_1(a)\oplus 0$ and $\alpha_2(a)=0\oplus \beta_2(a)$ for $a\in \cA_n$. This extends to an $\mathbb{F}_2^+$-action on $\widetilde{\cA}$, by sending generator $e_i$ to $\alpha_i$, $i=1,2$. By Proposition \ref{prop.LCM.N}, the resulting dynamical system is a right LCM dynamical system since ranges of $\alpha_i$ are both ideals in $\widetilde{\cA}$ and their ranges are orthogonal to one another. 

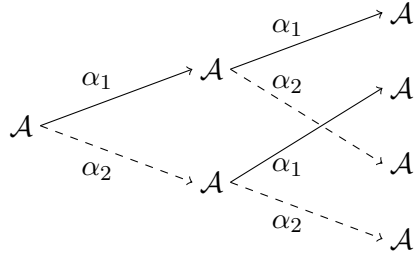
\begin{figure}[h]
    \centering

    \begin{tikzpicture}

	\draw[->] (1.25,0) -- (3.25,0.75);
	\draw[dashed,->] (1.25,0) -- (3.25,-.75);
	
	\draw[->] (3.75,.75) -- (5.75,1.5);
	\draw[dashed,->] (3.75,.75) -- (5.75,-0.5);
	\draw[->] (3.75,-.75) -- (5.75,0.5);
	\draw[dashed, ->] (3.75,-.75) -- (5.75,-1.5);

	% gamma labels
    \node at (1,0) {$\mathcal{A}$};
    
    \node at (3.5,.75) {$\mathcal{A}$};
    \node at (3.5,-.75){$\mathcal{A}$};
    
    \node at (6,1.5){$\mathcal{A}$};
    \node at (6,0.5){$\mathcal{A}$};
    \node at (6,-0.5){$\mathcal{A}$};
    \node at (6,-1.5){$\mathcal{A}$};
    
    \node at (2, 0.6) {$\alpha_1$};
     \node at (2, -0.6) {$\alpha_2$};
     
     \node at (4.5, 0.55) {$\alpha_2$};
     \node at (4.5, -0.55) {$\alpha_1$};
     \node at (4.5, 1.3) {$\alpha_1$};
     \node at (4.5, -1.3) {$\alpha_2$};
    \end{tikzpicture}
    \label{fig:Fk}
\end{figure}

Let $\phi_1=\phi:\cA_1\to\bh{H}$, and recursively define 
$\phi_{n+1}:\cA_{n+1}=\cA_n\oplus \cA_n\to\bh{H}$ by 
\[\phi_{n+1}(a\oplus b) = T_1 \phi_n(\beta_1^{-1}(a)) T_1^* + T_2 \phi_n(\beta_2^{-1}(b)) T_2^*.\]
Since $\phi_1$ is unital completely positive, the map $a\mapsto T_i \phi(\beta_i^{-1}(a))T_i^*$ is completely positive, $\phi_2$ is also a unital completely positive map. Inductively, each $\phi_n$ is a unital completely positive map. Notice that 
\begin{align*}
\phi_2(\varphi_1(a)) &= \phi_2(a\oplus a) \\
&= T_1 \phi(\beta^{-1}(a)) T_1^* + T_2 \phi(\beta^{-1}(a)) T_2^* \\
&= \phi_1(a),
\end{align*}
and inductively, $\phi_{n+1}(\varphi_n(a))=\phi_n(a)$.

\begin{figure}[h]
    \centering

    \begin{tikzpicture}

	\draw[->] (1.25,-3) -- (3.25,-3);
	\draw[->] (3.75,-3) -- (5.75,-3);
	\draw[->] (6.25,-3) -- (8.25,-3);
	
	\draw[->] (1,-3.2) -- (1, -4.3);
	\draw[->] (3.5,-3.2) -- (3.5, -4.3);
	\draw[->] (6,-3.2) -- (6, -4.3);
	
	\draw[->] (1.5,-4.5) -- (3,-4.5);
	\draw[->] (4,-4.5) -- (5.5,-4.5);
	\draw[->] (6.5,-4.5) -- (8.25,-4.5);

	% gamma labels

     \node at (1, -3) {$\mathcal{A}_1$};
     \node at (3.5, -3) {$\mathcal{A}_2$};
     \node at (6, -3) {$\mathcal{A}_3$};
     \node at (8.8, -3) {$\cdots$};
     
     \node at (1, -4.5) {$\mathcal{B}(\mathcal{H})$};
     \node at (3.5, -4.5) {$\mathcal{B}(\mathcal{H})$};
     \node at (6, -4.5) {$\mathcal{B}(\mathcal{H})$};
     \node at (8.8, -4.5) {$\cdots$};
     
     \node at (2.25, -2.7) {$\varphi_1$};
     \node at (4.75, -2.7) {$\varphi_2$};
     
     \node at (1.7, -3.75) {$\phi_1=\phi$};
     \node at (3.75, -3.75) {$\phi_2$};
     \node at (6.25, -3.75) {$\phi_3$};
    \end{tikzpicture}
    \label{fig:Fk1}
\end{figure}

The maps $\phi_n$ can be extended to a unital completely positive map $\widetilde{\phi}$ on the inductive limit $\widetilde{\cA}$. Moreover, for each $a\in \cA_n$, 
\[\phi_{n+1}(\alpha_1(a))=\phi_{n+1}(\beta_1(a)\oplus 0)=T_1 \phi_n(a) T_1^*,\]
and,
\[\phi_{n+1}(\alpha_2(a))=\phi_{n+1}(0\oplus \beta_2(a))=T_2 \phi_n(a) T_2^*.\]
We have for all $a\in \widetilde{\cA}$, 
\[\widetilde{\phi}(\alpha_i(a)) = T_i \widetilde{\phi}(a) T_i^*.\]
We can view $(T_1,T_2)$ as a contractive representation $T$ of $\mathbb{F}_2^+$ by sending generator $e_i$ to $T_i$. The resulting pair $(\widetilde{\phi},T)$ is a contractive covariant representation of $(\widetilde{\cA},\mathbb{F}_2^+, \alpha)$. Therefore, by Theorem \ref{thm.dym.dilation}, it dilates to an isometric covariant representation $(\widetilde{\pi}, V)$ on some Hilbert space $\cK\supset\cH$. Define $\pi:\cA\to\bh{K}$ by $\pi(a)=\widetilde{\pi}(a)$ for all $a\in \cA=\cA_1\subset\widetilde{\cA}$. We have
\begin{align*}
\pi(a) &= \widetilde{\pi}(a) \\
&= \widetilde{\pi}(a\oplus a) \\
&= \widetilde{\pi}(\alpha_1(\beta_1^{-1}(a))) + \widetilde{\pi}(\alpha_2(\beta_2^{-1}(a))) \\
&= V_1 \widetilde{\pi}(\beta_1^{-1}(a)) V_1^* + V_2 \widetilde{\pi}(\beta_2^{-1}(a)) V_2^* \\
&= V_1 \pi(\beta_1^{-1}(a)) V_1^* + V_2 \pi(\beta_1^{-1}(a)) V_2^*. 
\end{align*}
Therefore, the pair $(\phi,T)$ gets dilated to the pair $(\pi,V)$, and the relation \[\phi(a)=T_1\phi(\beta_1^{-1}(a))T_1^*+T_2\phi(\beta_1^{-1}(a))T_2^*,\] is preserved as
\[\pi(a)=V_1\pi(\beta_1^{-1}(a))V_1^* + V_2\pi(\beta_1^{-1}(a))V_2^*.\] One should also notice that by setting $a=I$, one gets $V_1V_1^*+V_2V_2^*=I$. Therefore $V_1$ and $V_2$ have orthogonal ranges.
\end{example}

One can easily generalize Example \ref{ex.dilation.Fk} to $\mathbb{F}_k^+$ and obtain the following corollary.

\begin{corollary}\label{cor.dilation.Fk} Let $\phi:\cA\to\bh{H}$ be a unital completely positive map on a unital $\mathrm{C}^*$-algebra $\cA$. Let $\beta_1,\cdots,\beta_n$ be $\ast$-automorphisms of $\cA$. Suppose that $T_1,\cdots,T_k\in\bh{H}$ are contractions such that for any $a\in \cA$, 
\[\phi(a) = \sum_{i=1}^k T_i \phi(\beta_i^{-1}(a)) T_i^*.\]
Then there exists a $\ast$-homomorphism $\pi:\cA\to\bh{K}$ on a Hilbert space $\cK\supset\cH$ and isometries $V_1,\cdots,V_k\in\bh{K}$ with orthogonal ranges, such that 
\[\pi(a) = \sum_{i=1}^k V_i \pi(\beta_i^{-1}(a)) V_i^*\]
and $\pi, V$ are dilations of $\phi, T$, and $\cH$ is co-invariant for $V$. 
\end{corollary} 

\begin{example} Let $\phi:\mathbb{C}\to\bh{H}$ by $\phi(x)=xI$ and $\beta_i$ be identity on $\cA$. Let $T_1,\cdots, T_k\in\bh{H}$ such that $\sum_{i=1}^k T_i T_i^*=I$. Then it is clear that \[\phi(x) = \sum_{i=1}^k T_i \phi(x) T_i^*.\]
Corollary \ref{cor.dilation.Fk} implies that we can dilate $(\phi, T)$ to a $\ast$-homomorphism $\pi:\mathbb{C}\to\bh{K}$ and isometries $V_1,\cdots,V_k\in\bh{K}$ with orthogonal ranges, and 
\[\pi(x)=\sum_{i=1}^k V_i \pi(x) V_i^*.\]
But $\pi(x)=xI$ is the only choice for $\pi$, and  thus $\sum_{i=1}^k V_i V_i^*=I$. This is precisely  Popescu's dilation of row contractions. 
\end{example} 

\begin{example} Let $\phi:M_k\to M_k$ be the map $\phi(E_{i,j})=\delta_{i,j}E_{i,j}$. In other words, $\phi$ maps a matrix $A$ to its diagonal. One can easily verify that $\phi$ is unital completely positive map but $\phi$ is not multiplicative for any $k\geq 2$. Let $\beta_i$ be the identity map on $M_k$ for each $i$, and
define $T_i=E_{i,1}$. Then for each $A\in M_k$, $\phi(A)=\diag(A_{11},\cdots, A_{kk})$, and one can check that 
\[\phi(A)=\sum_{i=1}^k T_i \phi(A) T_i^*.\]
Corollary \ref{cor.dilation.Fk} implies that we can dilate $\tcb{(\phi, T_i)}$ into a $\ast$-homomorphism $\pi: M_k\to\bh{K}$ and isometries $V_i$ with orthogonal ranges, such that for each $A\in M_k$, 
\[\pi(A)=\sum_{i=1}^k V_i \pi(A) V_i^*.\]

\end{example}

\begin{example} Let $X$ be a compact Hausdorff space and $\mu$ be a probability measure on $X$. One can define a unital completely positive map $\phi: C(X)\to\mathbb{C}$ by $\phi(f)=\int_X fd\mu$. Pick $t_1,\cdots, t_k\in\mathbb{C}$ such that $\sum_{i=1}^k |t_i|^2=1$, and thus 
\[\phi(f) =\sum_{i=1}^k t_i \phi(f) \overline{t_i}.\]

Corollary \ref{cor.dilation.Fk} implies that there exists a $\ast$-homomorphism $\pi:C(X)\to\bh{K}$ and isometries $V_1,\cdots,V_k\in\bh{K}$ with orthogonal ranges. Here, $\cK=\mathbb{C}\oplus \cK_1$ and with respect to such decomposition, 
\[\pi(f) = \begin{bmatrix} \int_X f(x) d\mu & * \\ * & * \end{bmatrix},\]
and,
\[V_i = \begin{bmatrix} t_i & 0 \\ * & * \end{bmatrix}.\] 
Moreover, for each $f\in C(X)$, 
\[\pi(f) = \sum_{i=1}^k V_i \pi(f) V_i^*.\]
\end{example}

%\bibliographystyle{abbrv}
%\bibliography{semigroup}

\end{document}